\documentclass[reqno, 12pt]{amsart}

\usepackage{amsmath,graphicx}
\usepackage{float}

\usepackage[algoruled,lined,noresetcount,norelsize]{algorithm2e}

\usepackage{amssymb,amsthm}
\usepackage{mathrsfs}
\usepackage{mathabx}\changenotsign
\usepackage{dsfont}
\usepackage{xcolor}
\usepackage[backref]{hyperref}
\hypersetup{
    colorlinks,
    linkcolor={red!60!black},
    citecolor={green!60!black},
    urlcolor={blue!60!black}
}
\usepackage[T1]{fontenc}
\usepackage{lmodern}
\usepackage[babel]{microtype}
\usepackage[english]{babel}
\usepackage{geometry}
\usepackage{tikz}
\usetikzlibrary{calc,decorations.pathmorphing}
\pgfdeclarelayer{background}
\pgfdeclarelayer{foreground}
\pgfdeclarelayer{front}
\pgfsetlayers{background,main,foreground,front}

\numberwithin{equation}{section}
\numberwithin{figure}{section}

\usepackage[shortlabels]{enumitem}
\usepackage{todonotes}

\theoremstyle{plain}
\newtheorem{theorem}{Theorem}[section]
\newtheorem{definition}{Definition}
\newtheorem{corollary}[theorem]{Corollary}
\newtheorem{lemma}[theorem]{Lemma}
\newtheorem{remark}[theorem]{Remark}
\newtheorem{claim}[theorem]{Claim}
\newtheorem{conjecture}[theorem]{Conjecture}
\newtheorem{question}[theorem]{Question}

\renewcommand{\epsilon}{\varepsilon}
\newcommand\eps{\varepsilon}

\newcommand{\Prob}{\mathbb{P}}
\newcommand{\BigT}{\operatorname{Big}_T}
\newcommand{\GoodT}{\operatorname{Good}_{T}}
\newcommand{\calE}{\mathcal{E}}
\newcommand{\calF}{\mathcal{F}}
\def\deg{\text{deg}}
\newcommand{\WC}{\textrm{WC}}
\newcommand{\cF}{\mathcal{F}}
\newcommand{\cH}{\mathcal{H}}
\newcommand{\cX}{\mathcal{X}}

\makeatletter
\def\moverlay{\mathpalette\mov@rlay}
\def\mov@rlay#1#2{\leavevmode\vtop{   \baselineskip\z@skip \lineskiplimit-\maxdimen
	\ialign{\hfil$\m@th#1##$\hfil\cr#2\crcr}}}
\newcommand{\charfusion}[3][\mathord]{
#1{\ifx#1\mathop\vphantom{#2}\fi
	\mathpalette\mov@rlay{#2\cr#3}
}
\ifx#1\mathop\expandafter\displaylimits\fi}
\makeatother

\newcommand{\dcup}{\charfusion[\mathbin]{\cup}{\cdot}}

\begin{document}

\title[Creating spanning trees in Waiter-Client games]{%
	Creating spanning trees in Waiter-Client games}

\author[Grzegorz Adamski]{Grzegorz Adamski}
\author[Sylwia Antoniuk]{Sylwia Antoniuk}
\author[Ma\l gorzata Bednarska-Bzd\k{e}ga]{%
	Ma\l gorzata Bednarska-Bzd\k{e}ga}
\author[Dennis Clemens]{Dennis Clemens}
\author[Fabian Hamann]{Fabian Hamann}
\author[Yannick Mogge]{Yannick Mogge}

\thanks{The research of the fourth and sixth author is supported 
by Deutsche Forschungsgemeinschaft (Project CL 903/1-1). }

\address{Department of Discrete Mathematics, Faculty of Mathematics 
	and CS, Adam Mickiewicz University, Pozna\'n, Poland}
\email{grzegorz.adamski@amu.edu.pl, sylwia.antoniuk@amu.edu.pl, 
	mbed@amu.edu.pl}
\address{Institute of Mathematics, Hamburg University of Technology, 
	Hamburg, Germany.}
\email{dennis.clemens@tuhh.de, fabian.hamann@tuhh.de,
 	yannick.mogge@tuhh.de}

\pagestyle{plain}

\maketitle

\begin{abstract}
For a positive integer $n$ and a tree $T_n$ on $n$ vertices, we consider an unbiased Waiter-Client game $\WC(n,T_n)$ played on the complete graph~$K_n$, in which Waiter's goal is to force Client to build a copy of $T_n$. We prove that for every constant $c<1/3$, if $\Delta(T_n)\le cn$ and $n$ is sufficiently large, then Waiter has a winning strategy in $\WC(n,T_n)$. On the other hand, we show that there exist a positive constant $c'<1/2$  and a family of trees $T_{n}$ with $\Delta(T_n)\le c'n$ such that Client has a winning strategy in the $\WC(n,T_n)$ game for every $n$ sufficiently large. We also consider the corresponding problem in the Client-Waiter version of the game.
\end{abstract}


\section{Introduction}\label{sec:intro}

We consider finite, two-person, perfect-information games called unbiased Waiter-Client games, which are closely related to Maker-Breaker and Avoider-Enforcer positional games (cf.~\cite{beck2008combinatorial}, \cite{hefetz2014positional}).
Given a hypergraph $\cH = (\cX,\cF)$, in the Waiter-Client game played on board $\cH$, 
the first player, called Waiter, offers the second player, called Client, two previously unclaimed elements of $\cX$. Client then chooses one of these elements which he keeps and the other element goes to Waiter. If in the final round only one unclaimed element remains, then it goes to Waiter. The game ends when all elements of $\cX$ have been claimed. Waiter wins if at the end there is a set $F \in \cF$ with all its elements claimed by Client; otherwise Client is the winner.

The rules of the Client-Waiter game on $\cH$ are similar, with the difference that if at the end of the game there is a set $F \in \cF$ with all its elements claimed by Client, then Client is the winner. Also the last round rule in the Client-Waiter game is different: if only one unclaimed element remains, it goes to Client. Let us add that Waiter-Client and Client-Waiter games were introduced by Beck~\cite{beck2002secmom} under the names of Picker-Chooser and Chooser-Picker, respectively.

Usually, Waiter-Client (or Client-Waiter) games are studied in a graph setting, i.e.~it is assumed that $\cX$ is the set of edges of the complete graph~$K_n$ and Waiter tries to force Client to build (or Client wants to build) a graph having a given graph property (which is just a family of graphs). In other words, the family $\cF$ consists of the edge-sets of a given graph family. With some abuse of notation, we say that such games are played on $K_n$.  

There is an interesting relation between games played on $K_n$ and random processes. For instance, the case where in a Waiter-Client game on $K_n$ Waiter plays randomly is the well-known Achlioptas process (without replacement). 
Furthermore, a classic observation shows that sometimes the winner of the game could be predicted using a heuristic known as the \emph{probabilistic intuition}. This intuition suggests that the result of the game is related to the properties of the (binomial) random graph $G(n, 1/2)$, in which the number of edges is roughly the same as the number of Client's edges at the end of the game. 
A fascinating example of this phenomenon was observed by Beck~\cite{beck2008combinatorial} who proved that the size of the largest clique that Waiter can force in a Waiter-Client game played on $K_n$ is $2\log_2n-2\log_2\log_2n+O(1)$. The same holds for the Client-Waiter version of the problem, i.e.~this is the size of the largest clique Client can build in $K_n$. Note that $2\log_2n-2\log_2\log_2n+O(1)$ is also the clique number of the random graph $G(n, 1/2)$ with high probability, see e.g.~\cite{bollobas2001random}. We return to random intuition hints in the last part of this section. 

In this paper, we focus on Waiter-Client games on $K_n$ where Waiter tries to force Client to build a copy of a~fixed spanning tree $T_n$. We denote such a game by $\WC(n,T_n)$. Throughout the paper we assume that every edge claimed by Client is colored red, while every edge in the Waiter's graph is colored blue. 
Note that if we relax the assumption that the spanning tree is fixed and we allow any spanning tree, then Waiter's task becomes much easier. Indeed, Csernenszky et al.~\cite{csernenszky2009chooser} showed that Waiter can achieve such a goal playing on any graph $G$ (not necessarily complete), if and only if $G$ contains two edge-disjoint spanning trees. Hence, Waiter can force a spanning red tree in $K_n$ for every $n\ge 4$. 

Returning to $\WC(n,T_n)$, we are interested in the outcome of the game depending on the maximum degree of $T_n$ and for large $n$. The last three authors and their co-authors proved in~\cite{clemens2020fast} that if $n$ is large enough, Waiter can force a red Hamilton path within $n-1$ rounds. Furthermore, Waiter can force a red copy of any fixed tree $T_n$ with $\Delta(T_n) \leq c \sqrt{n}$ within $n$ rounds, for some suitable constant $c$. 
Though the authors of~\cite{clemens2020fast} focused on the problem of building a given tree fast, they also posed the following question regarding maximum degrees.

\begin{question}\label{question:deg}
What is the largest integer $D(n)$ such that for every tree $T_n$ with $n$ vertices and the maximum degree at most $D(n)$ Waiter has a winning strategy in $\WC(n,T_n)$? 
\end{question}

The above mentioned result in~\cite{clemens2020fast} implies that $D(n)=\Omega(\sqrt n)$. On the other hand, it is known that $D(n)\le n/2+O(\sqrt{n\log n})$. Indeed, Waiter cannot even force a red star of size $n/2+c\sqrt{n\log n}$ for a sufficiently large constant $c>0$, which follows from a much more general results on discrepancy games played on hypergraphs, proved by Beck (\cite{beck2008combinatorial}, Theorem 18.3). Our main result is that $n/3+o(n)\le D(n)\le (1/2-c)n+o(n)$ for some positive constant $c$. We state it as the following two theorems.

\begin{theorem}\label{thm:giventree}
For every $\eps\in\left(0,\frac{1}{3}\right)$ there exist positive constants $b$ and $n_0$ such that the following holds. 
Let $T_n$ be a tree on $n\geq n_0$ vertices with $\Delta(T_n)<\left(\frac13-\eps\right)n$. Then Waiter has a winning strategy in $\WC(n,T_n)$. Furthermore, she can obtain her goal within at most $n+b\sqrt n$ rounds.
\end{theorem}


\begin{theorem}\label{thm:upperbound.wc}
There are positive constants $c$ and $n_0$ such that the following holds for every $n\geq n_0$. There exists a tree $T_n$ with $n$ vertices and $\Delta(T_n)<\left(\frac12-c\right)n$ such that Client has a~winning strategy in $\WC(n,T_n)$.
\end{theorem}

We prove the above theorem with $c=0.001$ and make no further effort to optimize it. 
Though the improvement from $D(n)\le 0.5 n+o(n)$ to $D(n)\le 0.499 n +o(n)$ seems cosmetic, we think that it is important for predicting a proper constant $C$ in the desired formula $D(n)=Cn+o(n)$, provided such a constant $C$ exists. The linear bounds on $D(n)$ prove that $D(n)$ is far from $\Theta(n/\log(n))$ suggested by the random graph $G(n,1/2)$ behavior~\cite{krivelevich2010embedding}. Nonetheless, our argument for Theorem~\ref{thm:upperbound.wc} involves randomness. Namely, we describe a randomized strategy for Client that allows avoiding a~red copy of a~given tree with high probability.  

Finally, let us mention the Client-Waiter version of the above game. In contrast to Waiter-Client games, here the random graph intuition may help. It is known \cite{komlos2001} that there exists a tree $T_n$ with $n$ vertices and $\Delta(T_n)=O(n/\log(n))$ such that with high probability the random graph $G(n,1/2)$ has no copy of $T_n$. We prove that a similar phenomenon occurs in Client-Waiter games.

\begin{theorem}\label{thm:upperbound.cw}
There are positive constants $c$ and $n_0$ such that the following holds. 
For every $n\geq n_0$ there exists a tree $T_n$ with $n$ vertices and $\Delta(T_n)\leq \frac{cn}{\log(n)}$ 
such that in a Client-Waiter game on $K_n$, Waiter can prevent Client from claiming a red copy of $T_n$.
\end{theorem}

\subsection*{Organization of the paper}

In Section~\ref{sec:prelim} we collect useful probability tools, positional games theory tools and some results on embedding trees. In Section~\ref{sec:rooted.forests}
we prove a few lemmas and theorems regarding forcing red forests in Waiter-Client games. 
Some of them may be of independent interest. For example we show (see Theorem~\ref{thm:root1del}) that if $n\ge 2$ and $F$ is a forest with at most $n$ vertices and exactly $m$ edges, such that every of its components has less than $m/3$ edges, then Waiter can force a~red copy of $F$ in~$K_n$. 

Theorem~\ref{thm:giventree} is proved in Section~\ref{sec:proof.linear.degree}, while
Theorem~\ref{thm:upperbound.wc} is proved in Section~\ref{sec:avoiding}.
Section~\ref{sec:cw_trees} contains the proof of Theorem~\ref{thm:upperbound.cw}.
We add some concluding remarks at the end of the paper.



\section{Preliminaries}\label{sec:prelim}

\subsection{Notation}

First of all, we set $[n]:=\{k\in\mathbb{N}:~ 1\leq k\leq n\}$ for
every positive integer $n$. 

Most of our graph notation is standard 
and follows~\cite{west2001introduction}.
Let $G$ be any graph. Then we write $V(G)$ and $E(G)$ for the vertex set and the edge set of $G$, respectively, and we set $v(G) :=|V(G)|$ and $e(G):=|E(G)|$.
If $\{v,w\}$ is a pair of vertices, we write $vw$ for short. 
The neighborhood of a vertex $v$ in graph $G$ is $N_G(v) : =\{w\in V(G): vw\in E(G)\}$,
and its degree is $\deg_G(v) :=|N_G(v)|$.
The maximum degree in $G$ is
$\Delta(G) := \max_{v\in V(G)} \deg_G(v)$.
Given any $A,B\subseteq V(G)$ and $v\in V(G)$,
we set
$E_G(A):=\{vw\in E(G):~ v,w\in A\}$,
$e_G(A):=|E_G(A)|$,
$E_G(A,B):=\{vw\in E(G):~ v\in A,w\in B\}$,
$e_G(A,B):=|E_G(A,B)|$,
$\deg_G(v,A):=|N_G(v)\cap A|$, and
$N_G(A):= \left(\bigcup_{v\in A} N_G(v)\right)\setminus A$.
Note that whenever the graph $G$ is clear from the context,
we may omit the subscript $G$ in all definitions above.
Given $A\subset V(G)$, we let $G[A]=(A,E_G(A))$  
be the subgraph of $G$ induced by $A$,
and we set $G-A:=G[V(G)\setminus A]$.
Similarly, if we have disjoint sets $A,B\subset V(G)$, we  write $G[A,B] := (A\cup B,E_{G}(A,B))$.
If $v$ is a vertex in $G$, we shortly write
$G-v$ for $G-\{v\}$.
Given  $B\subseteq E(G)$, we also let
$G\setminus B := (V(G),E(G)\setminus B)$. 
The disjoint union of graphs $G$ and $G'$ is denoted by $G\dcup G'$.
Moreover, an embedding of a graph $H$ into a graph $G$ is an injective map 
$f:V(H)\rightarrow V(G)$ such that $vw\in E(H)$ implies
$f(v)f(w)\in E(G)$.

We will consider forests $F$ with roots, where we may
allow that the components have up to two roots.
We say that a component of a forest is rooted if it contains exactly one root, 
and we say that it is double-rooted if it has exactly two roots.
By a rooted forest we mean a forest such that each of its components has exactly one root.
A leaf of $F$ is a vertex of degree 1 which is not a root (if $F$ has no roots, then every vertex of degree 1 is a leaf).  
With $L(F)$ we denote the set
of leaves in $F$. 
A bare path in $F$ is a path in $F$ such that all its inner vertices 
have degree 2 in $F$.
A leaf matching in $F$ is a matching such that each of its edges 
contains a vertex from $L(F)$.

Assume that some Waiter-Client game is in progress.
We let $W$ and $C$ denote the graphs consisting of Waiter's (blue) edges and 
Client's (red) edges, respectively, with the vertex set being equal 
to the graph that the game is played on.
If an edge belongs to $C\cup W$, we say that it is colored, or 
alternatively we say that it is claimed.
Otherwise, we say that the edge is uncolored or free. A vertex is free if all edges incident to it are free. 
We say that Waiter can force a graph $G$ withing $t$ rounds if Waiter has a strategy such that after at most $t$ rounds there is a red copy of $G$ at the board. 

In some of our arguments Waiter forces
a red copy of a given forest $F$ with roots in a greedy kind of way. 
That is, there is a sequence 
$F_1\subset F_2\subseteq \ldots \subseteq F_s=F$
with $F_1$ containing the roots of $F$, and for each $i$,
Waiter makes sure that after a total of $e(F_i)$ rounds
a red copy $\bar{F_i}$ of $F_i$ is created. 
Implicitly, this way we obtain an embedding 
$f_i:V(F_i)\rightarrow V(\bar{F_i})$ which extends the 
earlier embedding $f_{i-1}$ and which fixes the images of
the the vertices of $F_i$. 
Having such a strategy, we sometimes simply say that 
Waiter embeds the graph $F_i$ into the given board.
Moreover, if the image of a vertex $v$ under $f_i$ 
is a vertex $x$, then we simply say that Waiter maps $v$ to $x$.

\subsection{Probabilistic tools}

In our probabilistic argument, we will use
a variant of the Chernoff bound, for a sum of independent random variables $Z_i$ such that the number of terms is a random variable, not necessarily independent of $Z_i$.

\begin{lemma}\label{lemma:Chernoff_variant}
For $n\in{\mathbb N}$ and $\rho\in(0,1)$, let $(Z_i)_{i=1}^{n}$ be a sequence of i.i.d.~random variables with probability distribution
\begin{align*}
    Z_i = \begin{cases}
        1, & \text{with probability } \rho,\\
        0, & \text{with probability } 1-\rho.
    \end{cases}
\end{align*} 
Then for any random variable $T$ 
taking values in the set $\{0,1,\dots,n\}$, we have
\begin{align*}
    \Prob\left(\left|\sum_{i=1}^T Z_i - T\rho \right| > 2\sqrt{n\log n} \right) = o(n^{-2}).
\end{align*}
\end{lemma}

\begin{proof}
Using the union bound and the Chernoff inequality (see e.g.~\cite{janson2011random}), we obtain
\begin{align*}
     \Prob\left(\left|\sum_{i=1}^T Z_i - T\rho \right| > 2\sqrt{n\log n} \right) & \leq \sum_{k=1
}^n \Prob\left(\left|\sum_{i=1}^k Z_i - k\rho \right| > 2\sqrt{n\log n} \right) \leq \\
     & \leq \sum_{k=1}^n 2\exp\left(\frac{-8n\log n}{k}\right)  \leq 2n\exp(-8\log n) = o(n^{-2}). 
\end{align*}
\end{proof}

We say that an event, depending on $n$,
holds asymptotically almost surely (a.a.s.)
if it holds with probability tending to 1 if $n$
tends to infinity.

\subsection{Positional games tools}
We will use a variant of the Erd\H{o}s-Selfridge Breaker's winning criterion (see e.g.~\cite{beck2008combinatorial, hefetz2014positional}), adapted to the Client-Waiter version by the third author, which can be also applied for Waiter-Client games.

\begin{theorem}[Corollary 1.4 in~\cite{bednarska2013weight}]
\label{thm:WC_transversal}
 Consider a Waiter-Client game on a hypergraph $(\cX,\cF)$ satisfying
	$$\sum_{F\in \cF} 2^{-|F| + 1} < 1.$$
	Then Waiter has a strategy to force Client to
	claim at least one element in each of the
	sets in $\mathcal{F}$. 
\end{theorem}


\section{Forcing forests}\label{sec:rooted.forests}

For the proof of Theorem~\ref{thm:giventree} we aim to provide 
Waiter's strategy that forces Client to build any fixed $n$-vertex tree $T$ whose 
maximum degree is at most $\left(\frac{1}{3}-\eps\right)n$. For this, we will 
distinguish trees by certain structural properties,
but in any case, the overall strategy for Waiter will then be to first 
force a red copy of a subforest $F\subseteq T$ which contains all vertices of
large degree in $T$, and then she will complete this subforest 
to a copy of $T$ while making use of the structural properties 
mentioned above. Note that in such a procedure, when $F$ is already 
embedded, we also fix the images of some of the vertices in order to be able to embed
the remaining forest $T\setminus E(F)$. This is why we need to 
study games in which Waiter wants to force rooted forests. 
Below we prepare the tools for our main strategy, which is given in 
Section~\ref{sec:proof.linear.degree}, by collecting and proving
several lemmas regarding forcing different kinds of forests and forests with roots. 
A rich collection of lemmas is used to prove Theorems \ref{thm:root1del} and \ref{thm:doubleroot3} -- only these two results, together with Lemma~\ref{lem:doubleroot1v}, are applied in the proof of Theorem~\ref{thm:giventree}.

\subsection{Forcing Hamilton paths and perfect matchings}

The following two lemmas are more quantitative versions of results from \cite{clemens2020fast}
on forcing Hamilton paths with fixed endpoints. 

\begin{lemma}\label{lem:Ham.path.simple}
	For $n\geq 5$, Waiter can force Client to build a red Hamilton path on $K_n$ 
	within $n-1$ rounds and such that one of its endpoints 
	is incident with at most two blue edges. 
\end{lemma}

\begin{proof}
Let $v$ be any vertex of $K_n$. In each of the first two rounds Waiter offers any two free edges incident with $v$. The result is a red path $P$ on three vertices whose endpoints 
are not incident with any other colored edge outside the path. 
Next, in each of the following $n-3$ rounds, let $P$ be the 
current red path, and let $w$ be any vertex not in $P$. 
Then Waiter offers the two edges between $w$ and the endpoints of $P$,
hence extending the red path $P$ by one vertex.
By the end of round $n-1$, the path $P$ is a Hamilton path.
Let $w$ be the vertex which was added in the last round.
Then there are at most two blue edges incident with $w$:
the blue edge that was offered in the last round, and maybe
one more edge if $vw$ was offered within the first two rounds.
\end{proof}

\begin{lemma}\label{lem:Ham.path.ends.fixed}
	Let $x,y\in V(K_n)$ with $n\geq 8$. 
	Then within $n$ rounds Waiter can force Client to build a red Hamilton path 
	between $x$ and $y$ in $K_n$.
\end{lemma}

\begin{proof}
Using Lemma~\ref{lem:Ham.path.simple}, within $n-3$ first rounds
Waiter can force a red Hamilton path $P=(v_1,v_2,\ldots,v_{n-2})$ 
on $K_n-\{x,y\}$ in which $v_1$ is incident with at most
two blue edges. In round $n-2$ she offers the edges $xv_{n-2}$ and $yv_{n-2}$,
and without loss of generality we can assume that Client claims $v_{n-2}y$. Let $i,j$, with $3\leq i,j\leq n-2$, be such that $v_1v_i$ and $v_1v_j$ are free. This is possible since so far $v_1$ is incident with at most two blue edges. Then in round $n-1$ Waiter offers edges $v_1v_i$ and $v_1v_j$, and again without loss of generality we can assume that Client 
claims the edge $v_1v_i$. Finally, in round $n$, Waiter offers the edges
$xv_1$ and $xv_{i-1}$. It is easy to see that no matter which edge Client claims, 
a red Hamilton path is built as required.
\end{proof}

The next lemma shows that Waiter can force quickly a perfect matching in an almost complete bipartite graph (with equal bipartition).  

\begin{lemma}\label{lem:perfect_matching}
Let $t\ge 4$ and $K_{t,t}^{-}$ be the graph obtained from the complete bipartite graph $K_{t,t}$ by removing one edge.
By offering edges of $K_{t,t}^-$ only, Waiter can force a red copy of a perfect matching in $K_{t,t}^{-}$ within $t+1$ rounds.
\end{lemma} 

\begin{proof}

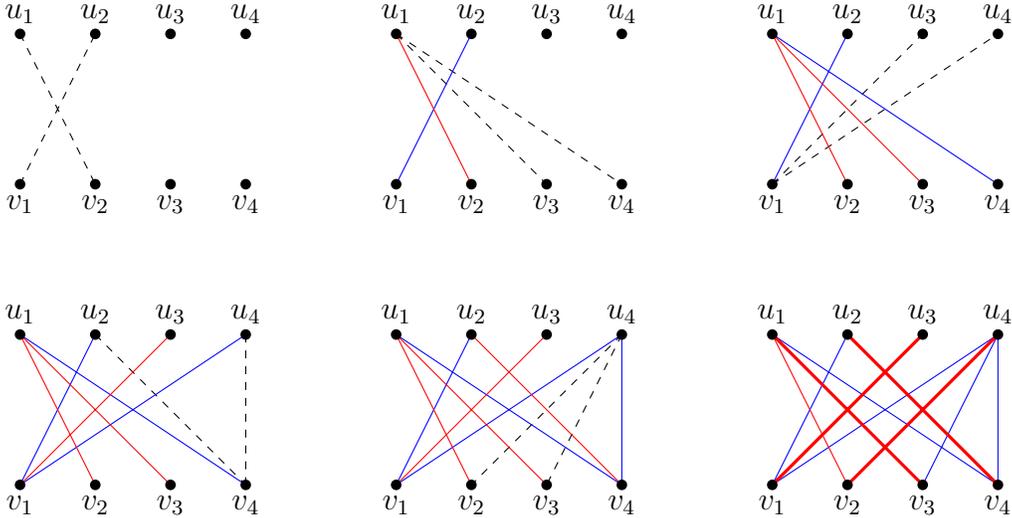
\begin{figure}[H]
    \begin{tikzpicture}
        \centering
        \coordinate (u1) at (0,1);
        \coordinate (u2) at (1,1);
        \coordinate (u3) at (2,1);
        \coordinate (u4) at (3,1);
        \coordinate (v1) at (0,-1);
        \coordinate (v2) at (1,-1);
        \coordinate (v3) at (2,-1);
        \coordinate (v4) at (3,-1);
        \begin{pgfonlayer}{front}
            \draw[dashed] (v1) -- (u2) (u1) -- (v2); 
            \foreach \i in {v1,v2,v3,v4,u1,u2,u3,u4} \fill (\i) circle (2pt);	
            \node at (0,1) [above] {$u_1$};
            \node at (1,1) [above] {$u_2$};
            \node at (2,1) [above] {$u_3$};
            \node at (3,1) [above] {$u_4$};
            \node at (0,-1) [below] {$v_1$};
            \node at (1,-1) [below] {$v_2$};
            \node at (2,-1) [below] {$v_3$};
            \node at (3,-1) [below] {$v_4$};
        \end{pgfonlayer}

        \coordinate (u1) at (5,1);
        \coordinate (u2) at (6,1);
        \coordinate (u3) at (7,1);
        \coordinate (u4) at (8,1);
        \coordinate (v1) at (5,-1);
        \coordinate (v2) at (6,-1);
        \coordinate (v3) at (7,-1);
        \coordinate (v4) at (8,-1);
        \begin{pgfonlayer}{front}
            \draw[dashed] (u1) -- (v3) (u1) -- (v4);
            \draw[blue] (v1) -- (u2);
            \draw[red] (u1) -- (v2); 
            \foreach \i in {v1,v2,v3,v4,u1,u2,u3,u4} \fill (\i) circle (2pt);	
            \node at (5,1) [above] {$u_1$};
            \node at (6,1) [above] {$u_2$};
            \node at (7,1) [above] {$u_3$};
            \node at (8,1) [above] {$u_4$};
            \node at (5,-1) [below] {$v_1$};
            \node at (6,-1) [below] {$v_2$};
            \node at (7,-1) [below] {$v_3$};
            \node at (8,-1) [below] {$v_4$};
        \end{pgfonlayer}

        \coordinate (u1) at (10,1);
        \coordinate (u2) at (11,1);
        \coordinate (u3) at (12,1);
        \coordinate (u4) at (13,1);
        \coordinate (v1) at (10,-1);
        \coordinate (v2) at (11,-1);
        \coordinate (v3) at (12,-1);
        \coordinate (v4) at (13,-1);
        \begin{pgfonlayer}{front}
            \draw[dashed] (v1) -- (u3) (v1) -- (u4);
            \draw[blue] (v1) -- (u2) (u1) -- (v4);
            \draw[red] (u1) -- (v2) (u1) -- (v3); 
            \foreach \i in {v1,v2,v3,v4,u1,u2,u3,u4} \fill (\i) circle (2pt);	
            \node at (10,1) [above] {$u_1$};
            \node at (11,1) [above] {$u_2$};
            \node at (12,1) [above] {$u_3$};
            \node at (13,1) [above] {$u_4$};
            \node at (10,-1) [below] {$v_1$};
            \node at (11,-1) [below] {$v_2$};
            \node at (12,-1) [below] {$v_3$};
            \node at (13,-1) [below] {$v_4$};
        \end{pgfonlayer}

        \coordinate (u1) at (0,-3);
        \coordinate (u2) at (1,-3);
        \coordinate (u3) at (2,-3);
        \coordinate (u4) at (3,-3);
        \coordinate (v1) at (0,-5);
        \coordinate (v2) at (1,-5);
        \coordinate (v3) at (2,-5);
        \coordinate (v4) at (3,-5);
        \begin{pgfonlayer}{front}	
            \draw[dashed] (u2) -- (v4) (u4) -- (v4);
            \draw[blue] (v1) -- (u2) (u1) -- (v4) (v1) -- (u4);
            \draw[red] (u1) -- (v2) (u1) -- (v3) (v1) -- (u3); 
            \foreach \i in {v1,v2,v3,v4,u1,u2,u3,u4} \fill (\i) circle (2pt);
            \node at (0,-3) [above] {$u_1$};
            \node at (1,-3) [above] {$u_2$};
            \node at (2,-3) [above] {$u_3$};
            \node at (3,-3) [above] {$u_4$};
            \node at (0,-5) [below] {$v_1$};
            \node at (1,-5) [below] {$v_2$};
            \node at (2,-5) [below] {$v_3$};
            \node at (3,-5) [below] {$v_4$};
        \end{pgfonlayer}

        \coordinate (u1) at (5,-3);
        \coordinate (u2) at (6,-3);
        \coordinate (u3) at (7,-3);
        \coordinate (u4) at (8,-3);
        \coordinate (v1) at (5,-5);
        \coordinate (v2) at (6,-5);
        \coordinate (v3) at (7,-5);
        \coordinate (v4) at (8,-5);
        \begin{pgfonlayer}{front}
            \draw[dashed] (u4) -- (v2) (u4) -- (v3);
            \draw[blue] (v1) -- (u2) (u1) -- (v4) (v1) -- (u4) (u4) -- (v4);
            \draw[red] (u1) -- (v2) (u1) -- (v3) (v1) -- (u3) (u2) -- (v4);
            \foreach \i in {v1,v2,v3,v4,u1,u2,u3,u4} \fill (\i) circle (2pt);
            \node at (5,-3) [above] {$u_1$};
            \node at (6,-3) [above] {$u_2$};
            \node at (7,-3) [above] {$u_3$};
            \node at (8,-3) [above] {$u_4$};
            \node at (5,-5) [below] {$v_1$};
            \node at (6,-5) [below] {$v_2$};
            \node at (7,-5) [below] {$v_3$};
            \node at (8,-5) [below] {$v_4$};
        \end{pgfonlayer}

        \coordinate (u1) at (10,-3);
        \coordinate (u2) at (11,-3);
        \coordinate (u3) at (12,-3);
        \coordinate (u4) at (13,-3);
        \coordinate (v1) at (10,-5);
        \coordinate (v2) at (11,-5);
        \coordinate (v3) at (12,-5);
        \coordinate (v4) at (13,-5);
        \begin{pgfonlayer}{front}
            \draw[blue] (v1) -- (u2) (u1) -- (v4) (v1) -- (u4) (u4) -- (v4) (u4) -- (v3);
            \draw[red] (u1) -- (v2);
            \draw[red, very thick] (u1) -- (v3) (v1) -- (u3) (u2) -- (v4) (u4) -- (v2);
            \foreach \i in {v1,v2,v3,v4,u1,u2,u3,u4} \fill (\i) circle (2pt);	
            \node at (10,-3) [above] {$u_1$};
            \node at (11,-3) [above] {$u_2$};
            \node at (12,-3) [above] {$u_3$};
            \node at (13,-3) [above] {$u_4$};
            \node at (10,-5) [below] {$v_1$};
            \node at (11,-5) [below] {$v_2$};
            \node at (12,-5) [below] {$v_3$};
            \node at (13,-5) [below] {$v_4$};
        \end{pgfonlayer}
    \end{tikzpicture}
\caption{\label{Fig:M_4vs.K_4,4}In order to build a perfect matching in $K_{4,4}^{-}$, in each round Waiter offers the two dashed edges.}
\end{figure}

We will use induction. First consider the case $t=4$.
Let $\{u_1,u_2,u_3,u_4\}$ and $\{v_1,v_2,v_3,v_4\}$ be the bipartition classes of $K_{4,4}^-$. Assume that $u_1v_1$ is the missing edge. Waiter plays in such a way that in each round Client's choices are irrelevant for the obtained structure of red and blue graphs, but they do have impact on the 
vertex labels in the edges offered by Waiter. 
(see Figure~\ref{Fig:M_4vs.K_4,4}).

Waiter starts by offering the edges $u_1v_2$ and $u_2v_1$. By symmetry we can assume that  $u_1v_2$ becomes red. Then Waiter offers the edges $u_1v_3$ and $u_1v_4$. Client chooses one of them, say $u_1v_3$. Then Waiter offers the edges $u_3v_1$ and $u_4v_1$. Again we can assume that Client colors for instance $u_3v_1$ red. Next Waiter offers the edges $u_2v_4$ and $u_4v_4$. Without loss of generality Client colors $u_2v_4$ red. Finally, Waiter offers the edges $u_4v_2$ and $u_4v_3$. Client colors any of them red, say $u_4v_2$, and we get the desired red matching, namely $\{u_1v_3, u_2v_4, u_3v_1, u_4v_2\}$, within 5 rounds.

Now for $t>4$, suppose that $uv$ is the missing edge. Then in the first round Waiter offers one edge incident with $u$, and the other incident with $v$, say $uv'$ and $vu'$. By symmetry we can assume that Client colors $uv'$ red. Then we remove from our graph the vertices $u$ and $v'$ and the blue edge $vu'$, and the problem reduces to forcing a perfect matching in $K_{t-1,t-1}^-$ (now $vu'$ is the missing edge), which by inductive assumption can be done within $t$ rounds. Altogether, we get a perfect matching in $K_{t,t}^-$ within $t+1$ rounds. 
\end{proof}

\subsection{Rooted forest-tuples}

We begin by defining a couple of technical notions used extensively throughout the remaining part of this section.

\begin{definition}

Let $F_1,F_2,\ldots,F_{\ell}$ be vertex-disjoint rooted forests (i.e.~each component has exactly one root), such that $e(F_i)>0$ for some $i$. We call the tuple $F=(F_1,F_2,\ldots,F_{\ell})$ a rooted forest-tuple,
and define $V(F) = \bigcup_{i\in[\ell]} V(F_i)$, $v(F) = |V(F)|$,  $E(F) = \bigcup_{i\in[\ell]} E(F_i)$, $e(F) = |E(F)|$, and
\begin{align*}
a(F) & = \max \{ e(F_i):~ i\in [\ell]\}, \\
k(F) & = |\{i\in [\ell]:~ e(F_i)=a(F)\}|, \\
t(F) & = |\{i\in [\ell]: e(F_i)>0\}|.   
\end{align*}

Moreover, we define the following:
\begin{enumerate}
\item[(i)] 
Let $\ell'\le \ell$. We call a rooted forest-tuple $F'=(F_1',F_2',\ldots,F_{\ell'}')$ a \textit{valid subforest} of $F$ if for every $i\in [\ell']$,
we have $F_i'\subseteq F_i$ and every root of $F_i'$ is a root of $F_i$.
Furthermore, we assume that a rooted forest with no edges consisting of $\ell$ roots of $F$ is also 
a valid subforest of $F$.  

\item[(ii)\label{def:forest_diff}] 
For a valid subforest $F'=(F_1',F_2',\ldots,F_{\ell}')$ of $F$, 
we define $F-F':=(F_1^\ast,F_2^\ast,\ldots,F_{\ell}^\ast)$ to be a~rooted forest-tuple in which for every $i\in [\ell]$, $F_i^\ast$ is a rooted forest such that $V(F_i^\ast)=V(F_i)$, $E(F_i^\ast)=E(F_i)\setminus E(F_i')$ and  the root set of $F_i^\ast$ is $V(F_i')$.
If  $e$ is an edge incident to a root in $F_j$, then by $F-e$ we mean the rooted forest-tuple $F-F'$, where
$F'_j$ has only one edge $e$ and the remaining forests of $F'$ have no edges.  

\item[(iii)] 
We say that $F$ is \textit{suitable} if  
$3a(F) + k(F) - e(F) \leq 2$ or if $a(F)=1$ and $t(F)\ge 4$.
\item[(iv)] 
We say that $F$ is $m$-suitable if it is a valid subforest of a suitable forest-tuple $F'$ with $e(F')=m$.
\item[(v)] 
We say that $F$ is \textit{simple} if $t(F)\ge 4$ and either $a(F)=2$ and $k(F)=1$, or $a(F)=1$.

\end{enumerate}
\end{definition}

Our main goal is to analyze whether Waiter can force a red copy of a forest $F$, based only on the sizes of the trees in $F$. We will show that if $F$ (or rather a forest-tuple made from $F$) is $m$-suitable, which roughly speaking means that every tree in $F$ has less than $m/3$ edges and $e(F)\le m$, then Waiter can force a red copy of $F$ on the board $K_{m+r}$, where $r=v(F)-e(F)$ is the number of roots. However, in some cases it is useful to have some additional properties under control, for example the number of blue edges between groups of red trees. This is why we do not simply divide a forest $F$ into trees, but rather into forests. Consequently, we will treat $F$ as a forest-tuple and consider sizes of the tuple forests instead of sizes of the trees. 

With a slight abuse of notation we will sometimes denote by $F$ also the rooted graph $F_1\dcup F_2 \dcup \ldots \dcup F_{\ell}$ corresponding to the rooted forest-tuple $F=(F_1,F_2,\dots,F_\ell)$.
Before we proceed, let us focus on some properties of suitable and $m$-suitable rooted forest-tuples.



\begin{lemma}\label{vier}
    If a rooted forest-tuple $F$  is suitable, then $t(F)\ge 4$.
\end{lemma}
\begin{proof}
    Suppose that $F$ is suitable and $t(F)\le 3$. Then $e(F) \ge 3a(F)+k(F)-2$. However, there are $k(F)$ forests with $a(F)$ edges and the remaining ones have at most $a(F)-1$ edges. Hence 
    \[e(F) \le k(F)a(F) + (3-k(F))(a(F)-1) = 3a(F)+k(F)-3\]
    and we reach a contradiction. 
\end{proof}

\begin{lemma}\label{suitable}
    If a rooted forest-tuple $F$ satisfies $a(F)<e(F)/3$, then $F$ is suitable.
\end{lemma}
\begin{proof}
    If $k(F)\le 3$, then $3a(F)+k(F)-e(F)\le 3a(F)+3-(3a(F)+1)=2$, so $F$ is suitable.
    Hence, suppose that $k(F)\ge 4$. Using the fact that $e(F)\ge a(F)k(F)$ and $1-a(F)\leq 0$, we get
    $$3a(F)+k(F)-e(F)\le 3a(F)+k(F)- a(F)k(F)=(k(F)-3)(1-a(F))+3\le 4-a(F).$$
    If $a(F)\ge 2$, then the inequality $3a(F)+k(F)-e(F)\le 2$ holds. Otherwise we have $a(F)=1$ and $t(F) = e(F) > 3$ by the assumption of the Lemma. Hence, in both cases $F$ is suitable.
\end{proof}


\begin{lemma}\label{nsuitform}
 Let $m\in{\mathbb N}$ and $F$ be a rooted forest-tuple. 
\begin{itemize}
\item
If $F$ is $m$-suitable, then $3a(F) + k(F) - m \leq 2$ or $a(F)=1$. 
\item
If $m\ge \max\{4,e(F)\}$ and either $3a(F) + k(F) - m \leq 2$ or $a(F)=1$, then $F$ is $m$-suitable. 
\end{itemize}
\end{lemma}
\begin{proof}
    Suppose first that $F$ is $m$-suitable. Then there exists a forest-tuple $F'\supseteq F$ with $m$ edges satisfying $3a(F')+k(F')-m\le 2$ or $a(F')=1$. If $a(F')= 1$, then $a(F)\le a(F')=1$ and we are done. 
    
    Hence assume that $3a(F')+k(F')-m\le 2$. If $a(F)=a(F')$, then $k(F)\le k(F')$, so $3a(F)+k(F)-m\le 3a(F')+k(F')-m\le 2$. Therefore suppose that $a(F)<a(F')$. If $k(F)\ge 4$, then $3a(F)<k(F)a(F)\le e(F)$, therefore, by Lemma~\ref{suitable}, $F$ is suitable and we have $3a(F)+k(F)-m\le 2$ or $a(F)=1$. Finally, if $k(F)<4$, then $3a(F)+k(F)-m\le 3a(F')+1-m\le 3a(F')+k(F')-m\le 2$ and we are done. 

    Now suppose that $F$ satisfies $3a(F) + k(F) - m \leq 2$ or $a(F)=1$, and let $F'$ be a forest-tuple obtained from $F$ by adding $m-e(F)$ forests, each consisting of a~single edge. 
    Then $F$ is a valid subforest of $F'$ so it is enough to show that $F'$ is suitable.
    Assume first that $3a(F) + k(F) - m \leq 2$ and $a(F)>1$. Then we have $a(F')=a(F)$, $k(F')=k(F)$ and $e(F')=m$, so $3a(F') + k(F') - e(F') = 3a(F) + k(F) - m \leq 2$, and hence $F'$ is suitable. On the other hand, if $a(F)=1$, then since $m\ge 4$, we have $t(F')\ge 4$, which implies that $F'$ is suitable.
\end{proof}

\begin{lemma}\label{lem:F-extended}
    Let $F$ be an $m$-suitable rooted forest-tuple. Let $F'$ be a rooted forest-tuple obtained from $F$ by adding $m-e(F)$ additional forests, each consisting of a single edge. Then $F'$ is suitable.
\end{lemma}

\begin{proof}
Notice that $m\ge e(F)$ and $a(F)=a(F')$.

First consider the case $a(F)=1$. Let $F''$ be a suitable rooted forest-tuple with $e(F'')=m$, given by $m$-suitability of $F$. It follows from Lemma~\ref{vier} that $t(F'')\geq 4$. Since $a(F')=1$ and $e(F')=m$, we have $t(F') \geq t(F'')$, and hence we get that $F'$ is suitable as well.

Now suppose that $a(F)>1$. We have $k(F)=k(F')$. Using Lemma~\ref{nsuitform} for $F$, we conclude that $3a(F')+k(F')-m=3a(F)+k(F)-m\le 2$. Thus $F'$ is suitable.
\end{proof}

\begin{lemma}\label{msuitable}
Let $F$ be a rooted forest-tuple and $m\in{\mathbb N}$.    
If $a(F)<m/3$ and $m\ge e(F)$, then $F$ is $m$-suitable.
\end{lemma}
\begin{proof}
    Let $F'$ be a rooted forest-tuple obtained from $F$ by adding $m-e(F)$ forests, each consisting of a single edge. In particular, if $e(F)=m$ then $F=F'$. We have $a(F')=a(F) < m/3 = e(F')/3$. Hence, by Lemma~\ref{suitable}, $F'$ is suitable. Since $F$ is a valid subforest of $F'$, the assertion follows.
\end{proof}

\subsection{Shrinking operations}

While considering a game in which Waiter is trying to obtain a red copy of a forest $F$ (or a rooted forest-tuple) by increasing its partial embedding, one can look at the game as a process of decreasing the part of $F$ left for embedding. Below we define two corresponding kinds of shrinking operations on rooted forest-tuples.

\begin{definition}\label{def:shrink}
Let $H=(H_1,H_2,\ldots,H_{\ell})$ be a rooted forest-tuple with the set of roots $R$.
Assume that $e(H_1)\ge e(H_2)\ge \ldots \ge e(H_{\ell})$. 
Let $B$ be a complete graph on at least $v(H)$ vertices, with fixed distinct images $u(r)$ of roots $r\in R$. 
\begin{enumerate}
\item
Suppose that $t(H)\ge 2$ and $r_i\in V(H_i)$, for $i=1,2$, are any non-isolated roots.
We say that playing on $B$, Waiter performs an $H$-shrink operation of type 1 if he plays one round in the following way. 
Waiter offers the edges $u(r_1)w$ and $u(r_2)w$ for any free vertex $w\in V(B)$ which is not an image of a root of $H$. 
\item
Suppose that $t(H)\ge 3$ and $r_i\in V(H_i)$, for $i=1,2,3$, are any non-isolated roots.
We say that playing on $B$, Waiter performs an $H$-shrink operation of type 2 if he plays two rounds in the following way. 
First Waiter offers the edges $u(r_1)w_1$ and $u(r_1)w_2$ for any distinct free vertices $w_1,w_2\in V(B)$ which are not root images. Client colors $u(r_1)w_i$ blue and $u(r_1)w_j$ red. In the following round Waiter offers $u(r_2)w_i$ and $u(r_3)w_i$.
\end{enumerate}

During the above $H$-shrink operation of type 1, Client claims a red edge $e$ incident to the image $u(r)$ of a root $r\in  V(H_{i})$, for some $i\in\{1,2\}$. Let $e'$ be any edge of $H$ incident to $r$. We say that the rooted forest-tuple $H-e'$ arises from the $H$-shrink operation of type $1$. Analogously, we define a rooted forest-tuple arising from the $H$-shrink operation of type 2, namely, it is $(H-e'_1)-e'_2$, where $e'_1,e'_2$ are edges of $H$ corresponding to the two red edges claimed by Client. 
\end{definition}

For simplicity of the above definition, we assumed the forests of $H$ are sorted by their number of edges. However, we will also allow $H$-shrink operations in case of any rooted forest-tuple -- in an $H$-shrink operation of type 1 Waiter deals with any two biggest forests of the tuple, while in an $H$-shrink operation of type 2 she deals with any three biggest forests. 

In the following lemma we prove among others that $H$-shrink operations of both types preserve suitability, provided $H$ is not simple. 

\begin{lemma} \label{lem:shrinksuit}
Let $H=(H_1,H_2,\ldots,H_{\ell})$ be a rooted forest-tuple with the set of roots $R$ and $s\in\{1,2\}$.
Suppose that $\min(k(H),2)=3-s$ and $t(H)>s$. 
Let $B$ be a complete graph on at least $v(H)$ vertices, with fixed distinct images $u(r)$ of roots $r\in R$. 
Assume that playing on $B$, Waiter performs an $H$-shrink operation of type $s$ and let $H'$ be the rooted forest-tuple arising from this operation. Then the following holds.
	\begin{enumerate}[(a)]
    \item\label{shrinksuit} If $H$ is suitable and not simple, then $H'$ is suitable.
    \item\label{shrinksuitm} If $H$ is $m$-suitable and $e(H)-t(H)\ge 2$, then $H'$ is $(m-s)$-suitable.
	\end{enumerate}	 
\end{lemma}

\begin{proof}
Since the order of forests in $F$ is irrelevant for the above properties, we can assume that $e(H_1)\ge e(H_2)\ge e(H_3)\ge \ldots \ge e(H_\ell)$. 
Suppose that $H$ is $m$-suitable and $e(H)-t(H)\ge 2$. Then $a(H)\ge 2$ and hence $3a(H)+k(H)-m\le 2$ in view of Lemma~\ref{nsuitform}. Another consequence of the inequality  $e(H)-t(H)\ge 2$ is that $3a(H)+k(H)\ge 8$, so $m\ge 6$. Thus, we also have $m-s\ge 4$. We know that $e(H')=e(H)-s$ and $m\ge e(H)$, hence $m-s\ge \max\{4,e(H')\}$. 
Therefore, in order to prove that $H'$ is $(m-s)$-suitable, it is enough to verify that $3a(H')+k(H')-(m-s)\le 2$ and use Lemma~\ref{nsuitform}. We consider two cases depending on $s$. 

\paragraph{Case 1} Suppose that $s=1$. 
Then $k(H)\ge 2$ by assumption that $\min(k(H),2)=3-s=2$, and hence $e(H_1)=e(H_2)=a(H)$. Observe that the $H$-shrink operation of type 1 decreases either $e(H_1)$ or $e(H_2)$ by 1, so $a(H')=a(H)$ and $k(H')=k(H)-1$. Thus
$$3a(H')+k(H')-(m-s)=3a(H)+k(H)-m\le 2.$$

\paragraph{Case 2} Suppose that $s=2$. 
Then $k(H)=1$ by assumption that $\min(k(H),2)=3-s=1$. The $H$-shrink operation of type 2 decreases $e(H_1)$ and exactly one of $e(H_2),e(H_3)$ by 1.  
We note that $a(H')=a(H)-1$ and consider two subcases, depending on the structure of $H'$. 
\begin{itemize}
\item 
If $k(H')\le 2$, then 
        $$3a(H')+k(H')-(m-s)\le 3(a(H)-1)+2-(m-2)=3a(H)+k(H)-m\le 2.$$
\item 
If $k(H')\ge 3$, then $e(H_2)=e(H_3)=\ldots=e(H_{k(H')+1})=a(H)-1$, so $a(H)+(a(H)-1)k(H')\le e(H)\le m$. This implies that
         \begin{align*}
        3a(H')+k(H')-(m-s)-2 & = 3(a(H)-1) + k(H') - m\\
        & \le 3a(H) - 3 + k(H') - a(H) - (a(H)-1)k(H') \\
        & = 1 - (a(H)-2)(k(H')-2).
    \end{align*}
    Since $e(H)-t(H)\ge 2$ and $k(H)=1$, we have $a(H)\ge 3$, so $1 - (a(H)-2)(k(H')-2)\le 0$. 
\end{itemize}

In both subcases $3a(H')+k(H')-(m-s)\le 2$ and the assertion \ref{shrinksuitm} follows. 

Property \ref{shrinksuit} is a consequence of \ref{shrinksuitm} applied with $m=e(H)$, since if $H$ is suitable and not simple, then $e(H)-t(H)\ge 2$.
\end{proof}

\begin{remark}
    If $F$ is suitable, then $t(F)\ge 4>2\ge s$. Therefore, in this case we do not need the assumption $t(F)>s$, required for performing a shrink operation of type $s$.
\end{remark}

Shrinking operations will be the main ingredient of Waiter's strategy in most games considered further in this chapter, but it would be more convenient to look at them in the reversed way, i.e.\ increasing the red graph. Below we define an operation of type $s$ corresponding to a shrink operation of type $s$. 

Let $F=(F_1,F_2,\ldots,F_{\ell})$ be a rooted forest-tuple with the set of roots $R$.
Consider a game played on a board $B$ which is the complete graph, and such that Waiter's aim is to force a red copy of $F$. Suppose that after a number of rounds Client's graph is a red copy $\bar F'=\bar F'_1 \dcup \bar F'_2 \dcup \ldots \dcup \bar F'_{\ell}$ of a valid subforest $F'=(F'_1, F'_2, \ldots, F'_{\ell})$ of $F$.
Let $R'$ be the set of roots of $F-F'$, and for every  $r\in R'$ let $u(r)$ be the image of $r$ in $\bar F'$.
We say that Waiter performs an \emph{operation of type $s\in\{1,2\}$}, if she performs an $(F-F')$-shrink operation of type $s$, according to Definition~\ref{def:shrink} applied with fixed vertices $u(r)$, $r\in R'$. During the operation of type 2, Client claims red edges $e_1,e_2$ incident to roots $r_1\in V(\bar F'_{i})$, $r_2\in V(\bar F'_{j})$, $i\neq j$, respectively. 
We define $\bar F''_{i}=\bar F'_{i}\cup\{e_1\}$,  $\bar F''_{j}=\bar F'_{j}\cup\{e_2\}$
and  $\bar F''_{k}=\bar F'_{k}$ for $k\notin\{i,j\}$, and we call $\bar F''=\bar F''_1 \dcup \bar F''_2 \dcup \ldots \dcup \bar F''_{\ell}$ the red graph arising from the operation of type 2. We define the red graph arising from the operation of type 1, in which $\bar F'$ increases by one edge, in an analogous way. 
Note that if $\bar F''$ arises from the operation of type $s$, then it is a copy of a valid subforest $F''$ of $F$ such that $F-F''$ arises from the $(F-F')$-shrink operation of type $s$.  

Observe that Waiter is able to perform operation of type $s\in\{1,2\}$ if Client's graph $\bar F'$ satisfies $t(F-F')> s$ and there are at least $s$ free vertices on the board which are not the roots of $\bar F'$. Furthermore, if Waiter is going to increase a red forest $\bar F_i$ by the operation of type 1, then she has freedom in deciding which edge of $F_i$, not yet embedded, will correspond to the next red edge $e$ in a copy of $F_i$, provided an endpoint of $e$ is mapped into a vertex of $\bar F'_i$. Similarly, while increasing $\bar F_i$ or $\bar F_j$ by the operation of type 2, Waiter is flexible in selecting (simultaneously) edges $e_i\in E(F_i), e_j\in E(F_j)$ such that one of them  will correspond to the next red edge in a copy of $F_i$ or $F_j$, as long as $e_i,e_j$ have exactly one endpoint mapped into a vertex of $\bar F_i,\bar F_j$, respectively. We summarize this observation in the following remark used later in the proofs.

\begin{remark}\label{rem:order}
Consider a game in which Waiter's aim is to obtain a red copy of a rooted forest-tuple $F=(F_1,F_2,\ldots,F_{\ell})$. Suppose she can force edges of $F_i$ in a series of (not necessarily subsequent) operations of type 1 or 2. Then the edges of $F_i$ can be forced by Waiter in any order, provided that 
at every stage of the game the part of $F$ already embedded is a valid subforest of $F$. Such an ordering of embedding edges can be fixed for every $F_i$ at the beginning of the game.  
\end{remark}

The result below is a straightforward consequence of the definitions of a valid subforest, the operation of type $s$ and Lemma~\ref{lem:shrinksuit}, so we omit its proof. 

\begin{lemma} \label{lem:add12}
Let $F=(F_1,F_2,\ldots,F_{\ell})$ be a rooted forest-tuple with its root set $R$ and
let $F'=(F'_1, F'_2, \ldots, F'_{\ell})$ be a valid subforest of $F$. 
Consider a moment of the game played on the complete graph $B$ such that Client's graph 
is a red copy $\bar{F'}=\bar{F}'_1 \dcup \bar F'_2 \dcup \ldots \dcup \bar F'_{\ell}$ of $F'$.
If Waiter is able to perform the operation of type $s\in\{1,2\}$, then the following holds for the red graph $\bar F''$ arising from this operation.
\begin{enumerate}[(a)]
\item
$\bar F''$ is a copy of a valid subforest $F''=(F''_1, F''_2, \ldots, F''_{\ell})$.
\item
No edge claimed during the operation of type $s$ has both endpoints in the roots of $F$.
\item
Every blue edge claimed during the operation of type $s$ lies between sets 
$V(\bar F''_{i})$ and $V(\bar F''_j)$ for some $i\neq j$.
\item
Every red edge  claimed during the operation of type $s$ is an edge of $\bar F''_i$, for some $i\in[\ell]$.  
\item\label{addsuit}
If $F-F'$ is suitable but not simple and $\min(k(F-F'),2)=3-s$, then $F-F''$ is suitable.
\item\label{addsuitm}
If $F-F'$ is $m$-suitable, $e(F)-t(F)\ge 2$, and $\min(k(F-F'),2)=3-s$, then $F-F''$ is $(m-s)$-suitable.
\end{enumerate}
\end{lemma}

\subsection{Forcing rooted forest-tuples}

We are ready to prove a couple of results regarding forcing rooted forests.

\begin{lemma} \label{lem:stage1}
Let $F=(F_1,F_2,\ldots,F_{\ell})$ be a suitable rooted forest-tuple with the set of roots $R$.
Let $B$ be a complete graph on $v(F)$ vertices, with fixed distinct vertices $u_1,u_2,\ldots, u_{|R|}$. 
Then, playing on $B$, Waiter can force a red copy $\bar{F'}=\bar{F}'_1 \dcup \bar F'_2 \dcup \ldots \dcup \bar F'_{\ell}$ of a valid subforest $F'=(F'_1, F'_2, \ldots, F'_{\ell})$ of $F$ such that $F-F'$ is suitable and simple.  Furthermore, she can achieve this goal within $e(F')$ rounds applying operations of type 1 or 2 and 
such that after the $j^{\text{th}}$ operation Client's graph is a red copy $\bar F^j=\bar F^j_1 \dcup \bar F^j_2 \dcup \ldots \dcup \bar F^j_{\ell}$ of a valid subforest $F^j=(F^j_1, F^j_2, \ldots, F^j_{\ell})$ of $F$
satisfying the following properties.
	\begin{enumerate}[(a)]
	\item\label{stage1roots} 
     For every $r\in R$, the image of $r$ in $\bar F^j$ is $u_r$.
    \item No edges between the images of the roots are colored.
    \item Every blue edge lies between the sets $V(\bar F^j_{i})$ and $V(\bar F^j_{i'})$, for some $i\neq i'$.
    \item Every red edge is an edge of $\bar F^j$. 
    \item\label{stage1suit} 
     $F-F^j$ is suitable.
	\end{enumerate}	 
\end{lemma}

\begin{proof}
Note that there are exactly $|R|$ trees which are the components of $F$. The general strategy for Waiter is to force Client to grow each tree starting from the root, until the obtained rooted forest-tuple is simple. Our argument is inductive.

At the beginning of the game, let $F_i^0$ consist of the roots in $R\cap V(F_i)$, 
and $\bar{F}_i^0$ consist of their fixed images $u_r$, $r\in R\cap V(F_i)$.
Then  
$\bar{F}^0 := (\bar{F}_1^0,\bar{F}_2^0,\ldots,\bar{F}_{\ell}^0)$
and $F^0 := (F_1^0,F_2^0,\ldots,F_{\ell}^0)$
satisfy \ref{stage1roots} -- \ref{stage1suit}. 

For $j\ge 0$, assume that $j$ operations of type 1 or 2 
were already played and that Client's graph
$\bar F_1^j \dcup \bar F_2^j \dcup \ldots \dcup \bar F_{\ell}^j$
is a copy of a rooted forest $F_1^j \dcup F_2^j \dcup \ldots \dcup F_{\ell}^j$
as described above. Assume further that
$F-F^j$ is not simple and hence Waiter is going to play
the $(j+1)^{\text{st}}$ operation of type 1 or 2. Note that $F-F^j$ is suitable, so $t(F-F^j)\ge 4$ and there are at least 4 free vertices among non-roots images on the board, so 
Waiter can perform an operation of any type. 
The type of the operation depends on $k(F-F^j)$.  
If $k(F-F^j)\ge 2$, then she performs the operation of type 1; otherwise 
she performs the operation of type 2. Let $\bar F^{j+1}$ be the red graph arising from this operation.
In view of Lemma~\ref{lem:add12} applied to $F$ and $F^j$, the graph $\bar F^{j+1}$ is a copy of
a valid subforest of $F$, let us denote it by $F^{j+1}$. Conditions \ref{stage1roots} -- \ref{stage1suit} are satisfied by $\bar F^{j+1}$ and $F^{j+1}$ in view of the inductive hypothesis and Lemma~\ref{lem:add12}.
In particular, the suitability of $F^{j+1}$ follows from Lemma~\ref{lem:add12}~\ref{addsuit}.

Observe that the number of edges of $F-F^j$ is decreasing with $j$, while $t(F-F^j)\ge 4$ by Lemma~\ref{vier}, so $F-F^{s}$ is simple for some $s\ge 0$. It is also suitable because of \ref{stage1suit}. Finally, note that every red edge claimed in the game becomes an edge of $\bar F^s$, so it takes Waiter $e(F^s)$ rounds in total to achieve the desired graph.  
\end{proof}

\begin{lemma} \label{lem:starind}
Let $F=(F_1,F_2,\ldots,F_{\ell})$ be a suitable rooted forest-tuple.
Then playing on a complete graph $B$ with $v(F)$ vertices, Waiter can force a copy $\bar F=\bar F_1 \dcup \bar F_2 \dcup \ldots \dcup \bar F_{\ell}$ of $F_1 \dcup F_2 \dcup \ldots \dcup F_{\ell}$ 
within $e(F)+1$ rounds and such that the following properties hold.
	\begin{enumerate}[(a)]
	\item\label{suitroots} 
    All roots are mapped into vertices fixed at the beginning.
    \item No edges between the images of the roots are colored.
    \item\label{suitblue}  
    For every $i\in[\ell]$ there are no blue edges spanned by vertices in $V({\bar F_i})$.
    \item\label{suitred} 
    There are no red edges spanned by vertices in $V({\bar F_i})$ other than those of $\bar F_i$.
	\end{enumerate}	 
\end{lemma}

\begin{proof}
Let $R$ be the set of all roots of
$F_1 \dcup F_2 \dcup \ldots \dcup F_{\ell}$, and for each $r\in R$
fix a~distinct vertex $u_r$ of $B$. We present Waiter's strategy in two stages. In the first stage she forces Client to build a copy of a valid subforest $F'$ of $F$ such that $F-F'$ is simple. In the second stage Waiter forces a copy of $F-F'$, which completes building a copy of $F$. 

\paragraph*{\textbf{Stage I}}

Waiter applies her strategy provided by Lemma~\ref{lem:stage1} and forces Client to build a red copy 
$\bar F'=(\bar F'_1, \bar F'_2, \ldots, \bar F'_{\ell})$ of a 
valid subforest $F'=(F'_1, F'_2, \ldots, F'_{\ell})$ of $F$ such that $F-F'$ is suitable and simple, and such that every root $u_r$ of $\bar F'$ is the image of $r$, for $r\in R$. It follows from Lemma~\ref{lem:stage1} that no edges between the roots  $u_r$ are colored, all red edges are edges of $\bar F'$ and every blue edge lies between the sets $V(\bar F'_{i})$ and $V(\bar F'_{j})$, for some $i\neq j$. The first stage lasts $e(F')$ rounds. 

\paragraph*{\textbf{Stage II}}

Let $F^\ast = (F^\ast_1,\ldots,F^\ast_{\ell}) := F-F'$.
We set $a_i = e(F^\ast_i)=e(F_i) - e(F'_i)$ for every $i\in [\ell]$.
By reordering forests $F'_1,\ldots,F'_{\ell}$,
we can assume that $a_1 \leq 2$ and 
$a_2=\ldots = a_{\tilde{\ell}} = 1$ with $\tilde{\ell} := t(F - F') \geq 4$,
since $F-F'$ is simple and suitable. 

First assume that $a_1=2$. 
Let $V(B)\setminus V(\bar F')=:\{v_1,v_2,\ldots,v_{\tilde{\ell}+1}\}$.
Notice that all vertices of this set are free because of the properties of $\bar F'$ mentioned at the end of Stage I.

Waiter now plays as follows.
For the first round of Stage II, let $w$ be the image of a root of $F^\ast_1$. 
Waiter offers $wv_{\tilde{\ell}}$ and 
$wv_{\tilde{\ell}+1}$. 
By symmetry, we can assume that
Client colors $wv_{\tilde{\ell}+1}$ red.

Afterwards, 
for every $i\in [\tilde{\ell}]$, let $w_i\in V(\bar F'_i)$
be the unique vertex at which still exactly one red edge needs to be forced.
Then Waiter's goal is to get a red perfect matching
between the sets $W:=\{w_1,w_2,\ldots,w_{\tilde{\ell}}\}$
and $V:=\{v_1,v_2,\ldots,v_{\tilde{\ell}}\}$.
Due to the previous round, in the following rounds Waiter cannot offer $w_1v_{\tilde{\ell}}$, since 
either $w_1=w$ and $w_1v_{\tilde{\ell}}$ is blue or $w_1\neq w$ and $w_1v_{\tilde{\ell}}$ is free but,
if colored red, would violate condition~\ref{suitblue}.   
Other edges between $W$ and $V$ are still free. 
Hence, using Lemma~\ref{lem:perfect_matching}, 
Waiter can force a required red matching within $\tilde{\ell}+1$ rounds.

Now assume that $a_1=1$. This case is similar, yet a bit simpler, since we can apply Lemma~\ref{lem:perfect_matching} at once.  

In both cases Waiter wastes one red edge 
when she tries to force a perfect matching, so Stage II lasts $e(F)-e(F')+1$ rounds.
Furthermore, this extra red edge joins copies of two distinct forests $F_i$ and $F_j$,
and every blue edge claimed in Stage II joins copies of two distinct forests as well.

Summarizing, Stage I and Stage II last together $e(F)+1$ rounds and
all properties \ref{suitroots} -- \ref{suitred} are satisfied.
\end{proof}

In some cases Waiter may want to force a forest that has a smaller number of vertices than the board $B$. We formulate two lemmas, corresponding to Lemma~\ref{lem:stage1} and Lemma~\ref{lem:starind}, with an extra condition on the blue edges incident to vertices not belonging to the red forest obtained at the end. 

\begin{lemma} \label{lem:stage1m}
Let $F=(F_1,F_2,\ldots,F_{\ell})$ be an $m$-suitable rooted forest-tuple with the set of roots $R$. Let $B$ be a complete graph on $v(F)+m-e(F)$ vertices, with fixed $\ell$ vertices $u_1,u_2,\ldots, u_\ell$. Then, playing on $B$, Waiter can force a red copy $\bar{F'}=\bar{F'_1} \dcup \bar{F'_2} \dcup \ldots \dcup \bar{F'_{\ell}}$ of a valid subforest $F'=(F'_1, F'_2, \ldots, F'_{\ell})$ of $F$ such that $F-F'$ is $(m-e(F'))$-suitable and either $t(F-F')\le 2$ and $k(F-F')=1$, or $e(F-F')-t(F-F')\le 1$. 
Furthermore, she can achieve this goal within $e(F')$ rounds applying operations of type 1 or 2, and 
such that after the $j^{\text{th}}$ operation, Client's graph is a red copy $\bar F^j=\bar F^j_1 \dcup \bar F^j_2 \dcup \ldots \dcup \bar F^j_{\ell}$ of a valid subforest $F^j=(F^j_1, F^j_2, \ldots, F^j_{\ell})$ of $F$
satisfying the following properties.
 	\begin{enumerate}[(a)]
	\item\label{stage1mroots} 
     For every $r\in R$, the image of $r$ in $\bar F^j$ is $u_r$.
    \item No edges between the images of the roots are colored.
    \item Every blue edge lies between the sets $V(\bar F^j_{i})$ and $V(\bar F^j_{i'})$, for some $i\neq i'$.
    \item Every red edge is an edge of $\bar F^j$. 
    \item\label{stage1msuit} 
     $F-F^j$ is $(m-e(F^j))$-suitable.
	\end{enumerate}	 
\end{lemma}
\begin{proof}
The argument is similar to the proof of Lemma~\ref{lem:stage1}, nonetheless we present the inductive argument (omitting a few details), since an $m$-suitable rooted forest-tuple may not be suitable. 

For $j\ge 0$, assume that $j$ operations of type 1 or 2 
were already played and that Client's graph
$\bar F_1^j \dcup \bar F_2^j \dcup \ldots \dcup \bar F_{\ell}^j$
is a copy of a rooted forest $F_1^j \dcup F_2^j \dcup \ldots \dcup F_{\ell}^j$
satisfying \ref{stage1mroots} -- \ref{stage1msuit}. Assume further that
$e(F-F^j)-t(F-F^j)\ge 2$, and $t(F-F^j)\ge 3$ or $k(F-F^j)\ge 2$, so Waiter is going to perform
the $(j+1)^{\text{st}}$ operation of type 1 or 2. If $k(F-F^j)\ge 2$, then Waiter performs the operation of type 1,
while if $k(F-F^j)=1$ and $t(F-F^j)\ge 3$, then the operation of type 2 is performed. 
Let $\bar F^{j+1}$ be the red graph arising from the operation of type $s$, $s\in\{1,2\}$, performed by Waiter.
In view of Lemma~\ref{lem:add12} applied to $F$ and $F^j$, the graph $\bar F^{j+1}$ is a copy of
a valid subforest of $F$. Let us denote it by $F^{j+1}$. Conditions \ref{stage1roots} -- \ref{stage1suit} are satisfied by $\bar F^{j+1}$ and $F^{j+1}$ in view of the inductive hypothesis and Lemma~\ref{lem:add12}.
In particular, the $(m-e(F^{j+1}))$-suitability of $F^{j+1}$ follows from Lemma~\ref{lem:add12}~\ref{addsuitm}, the assumption that $F-F^j$ is $(m-e(F^{j}))$-suitable, $e(F-F^j)-t(F-F^j)\ge 2$, and $m-e(F^{j})-s=m-e(F^{j+1})$.
This finishes the inductive argument.
\end{proof}

Let us recall that if a vertex $u$ of a rooted forest $F$ is a root, then we do not treat it as a leaf, even if $\deg_F(u)=1$. For example, an isolated edge in a rooted forest has only one leaf, not two. 

\begin{lemma}\label{starindmore}
Let $F=(F_1,F_2,\ldots,F_{\ell})$ be an $m$-suitable rooted forest-tuple and let $m>e(F)$.
Then playing on a complete graph $B$ on $m+v(F)-e(F)$ vertices, Waiter can force a red copy $\bar{F}=\bar{F}_1 \dcup \bar{F}_2 \dcup \ldots \dcup \bar{F}_{\ell}$ of $F_1 \dcup F_2 \dcup \ldots \dcup F_{\ell}$ 
within $e(F)$ rounds and such that the following properties hold.
	\begin{enumerate}[(a)]
	\item\label{mroots} All roots are mapped into vertices fixed at the beginning.
    \item No edges between the images of the roots are colored.
    \item For every $i\in[\ell]$, there are no blue edges spanned by vertices in $V({\bar F_i})$. 
    \item\label{mred} There are no red edges spanned by vertices in $V({\bar F_i})$ other than those of $\bar F_i$.
	\item\label{minter} All colored edges intersect $V(\bar F)$.
    \item\label{mleaf} Every vertex not in $V(\bar F)$ is incident with at most one colored edge, and the other endpoint of such an edge is not a leaf of $\bar{F}$.
	\end{enumerate}	 
\end{lemma}

\begin{proof}
We divide the game into two stages. In the first stage Waiter forces a red copy of a valid subforest $F'$ of $F$ such that $F-F'$ either has at most two non-empty forests or all forest in the tuple are very small. In the second stage, Waiter forces a copy of $F-F'$, which completes building a copy of $F$. 

\paragraph*{\textbf{Stage I}}

Waiter applies her strategy provided by Lemma~\ref{lem:stage1m} and forces Client to build a red copy 
$\bar F'=(\bar F'_1, \bar F'_2, \ldots, \bar F'_{\ell})$ of a 
valid subforest $F'=(F'_1, F'_2, \ldots, F'_{\ell})$ of $F$ such that $F^\ast:=F-F'$ is $(m-e(F'))$-suitable and either $t(F^\ast)\le 2$ and $k(F^\ast)=1$, or $e(F^\ast)-t(F^\ast)\le 1$. Furthermore, every root $u_r$ of $\bar F'$ is the image of $r$, for $r\in R$. It follows from Lemma~\ref{lem:stage1m} that no edges between the vertices $u_r$ are colored, all red edges are edges of $\bar F'$ and every blue edge lies between sets $V(\bar F'_{i})$ and $V(\bar F'_{j})$ for some $i\neq j$. The first stage lasts $e(F')$ rounds. 

\paragraph*{\textbf{Stage II}}

For $F^\ast = (F^\ast_1,\ldots,F^\ast_{\ell})$
we set $a_i = e(F^\ast_i)=e(F_i) - e(F'_i)$ for every $i\in [\ell]$.
By reordering the forests $F^\ast_1,\ldots,F^\ast_{\ell}$,
we can assume that $a_1 \geq a_2\geq\ldots \geq a_{\ell}$. 
We consider two cases depending on the properties of $F^\ast$ mentioned at the end of Stage I. 

\smallskip
\paragraph{Case 1} Suppose that $t(F^\ast)\le 2$ and $k(F^\ast)=1$. Then $a_1=a(F^*)$ and $a_2\le a(F^*)-1$ (in case $\ell=1$ we set $a_2=0$).

    If $a_1=1$, then Waiter can simply finish the game in one round, since 
    $v(B)=v(F)+m-e(F)\ge v(F)+1=v(F^\ast)+2$, and hence there are still two
    free vertices (which are not roots) at this moment. Here we used the assumption that $m>e(F)$. 
    
    If $a_1>1$, then in view of $(m-e(F'))$-suitability of $F^\ast$ and Lemma~\ref{nsuitform}, we have 
    $3a(F^\ast)+k(F^\ast)-(m-e(F'))\le 2$. Thus $m\ge 3a(F^\ast)-1+e(F')$ and the number of free vertices (which are not roots) is 
    $v(B)-v(F')=m-e(F)+v(F)-v(F')\ge 3a(F^*)-1\ge 2a_1+a_2$.

    For the next $a_2$ rounds, Waiter performs the operation of type 1. Let $\bar F''$ be the red graph arising after all $a_2$ operations. It follows from Lemma~\ref{lem:add12} and the properties of the colored graph at the end of Stage I that no edges between the vertices $u_r$ are colored, all red edges are edges of $\bar F''$ and every blue edge lies between distinct forests of $\bar F''$. Note that exactly $a_2$ free vertices where used in these $a_2$ rounds.
    
   The part of $F$ remaining for embedding has $a_1$ edges and since there are still at least $2a_1$ free vertices left, Waiter can simply use two free vertices for each remaining edge to force the desired red copy $\bar F$.
   Note that the only vertices not in $V(\bar{F})$ incident to a colored edge at the end of the game are the ones used during the last $a_1$ rounds. Every such vertex was used exactly once and the edge incident to that vertex was also incident to a~vertex which was already in the red forest and which had another red incident edge, so it was not a leaf. Therefore properties \ref{minter} and \ref{mleaf} are satisfied. Since the last $a_1$ rounds do not spoil properties \ref{mroots} -- \ref{mred}, we conclude that properties \ref{mroots} -- \ref{mleaf} hold. 

\smallskip
\paragraph*{Case 2} Suppose that Case 1 does not hold and $e(F^\ast)-t(F^\ast)\le 1$.

Then either $a_1=2$, $t(F^\ast)\ge 3$ and $a_2,a_3,\ldots, a_{\ell}\le 1$, or $a_1=1$. In the first case Waiter can perform the operation of type 2. The red graph $\bar F''$ arising after this operation satisfies $a(\bar F'')=1$ and further argument coincides with the analysis in case $a_1=1$. For the next $t(F^\ast)-1$ rounds, Waiter performs the operation of type 1. Thereby she forces a red matching of size $t(F^\ast)-1$ and only one edge of $F$ remains for embedding. Then Waiter can finish the game in one round, since there are two free vertices (which are not roots) at this moment. Since the operations of type 1 or 2 did not spoil properties \ref{mroots} -- \ref{mred}, and the last round did not spoil properties \ref{minter} and \ref{mleaf}, we obtain a desired red copy of $F$. 
\end{proof}
     
We are now ready to state and prove one of the main results of this section.

\begin{theorem}\label{thm:root1del}
	Let $0\leq q < t$ be integers, and let $F$ be a forest 
	with at most $t$ edges, consisting of $s$ rooted trees $T_1, T_2, \ldots, T_s$, and such that for each $i\in[s]$ we have $1\leq e(T_i) < \frac{t-q}{3}$.
    Let $r_1,\ldots,r_s$ be the roots of the corresponding trees.
	Consider a partially colored complete graph $B$ with $V(B) = U\dcup V$, $|U|=s$, $|V|=t$,
	and such that $B$ contains a collection of $s$ vertex-disjoint blue stars with centers $u_1, \ldots, u_s\in U$, with total number of edges equal $q\leq t$, 
	while all other edges of $B[U,V]$ and all edges of $B[V]$ are free.
	Then, playing on $B$, Waiter can force a red copy $\bar F$ of $F$ within at most 
	$e(F)+1$ rounds in such a way that the following holds.
	\begin{enumerate}[(a)]
	\item\label{root1del_roots} Every root $r_i$ is mapped into $u_i$.
	\item All colored edges of $\bar{F}$ intersect $V(\bar F)$.
	\item No edges connecting vertices from $U$ are offered by Waiter.
	\item\label{root1del_nonleaf} Every vertex not in $V(\bar F)$ is incident with at most one colored edge in $\bar{F}$, and the other endpoint of such an edge is not a leaf of $\bar F$.
	\end{enumerate}	  
\end{theorem}

\begin{proof}
We apply induction on $q$. 

For $q=0$ and $e(F)=t$ we have $a(F) < e(F)/3$. Hence, by Lemma~\ref{suitable}, $F$ is suitable, and by Lemma~\ref{lem:starind}, Waiter can force the desired copy of $\bar{F}$ in $B$.

For $q=0$ and $e(F)<t$ we also have $a(F)<t/3$. Hence, using Lemma~\ref{msuitable}, we get that $F$ is $t$-suitable. Then, by Lemma~\ref{starindmore}, Waiter can force the desired copy of $\bar{F}$ in $B$.

For $q>0$ we consider two cases, in both of which we start by mapping the roots $r_i$ to vertices $u_i$ in $U$.

\smallskip
\paragraph*{Case 1} $|U|\ge 3$.

Then there exists a vertex $v\in V$ which is incident to a blue edge and there are two root images $u_i,u_j\in U$, such that the edges $u_iv, u_jv$ are free. 
Waiter offers $u_iv,u_jv$. By symmetry we can assume that Client colors $u_iv$ red. Let $T_i$ be the component of $F$ with root $r_i$, and let $r_i'$ be any neighbor of $r_i$ in $T_i$.
We replace $F$ with $F':=F-r_ir_i'$, $U$ with $U':=U\cup\{v\}$ and $V$ with $V':=V\setminus\{v\}$, and add $r_i'$ to the set of roots of $F$
with a fixed image $v$. Moreover, if the new tree rooted at $r_i$ or $r_i'$ in $F'$ is trivial, we delete it from $F'$, and we delete the corresponding image of that root from $U'$ as well.
We are left with the forest $F'$ with 
$t'=t-1$ edges and a board $B':=B[U'\cup V']$ such that there are 
$q'=q-1$ blue edges forming stars with centers in $U'$ and edges towards $V'$.
As every component of $F'$ has less than $\frac{t-q}{3}=\frac{t'-q'}{3}$
edges, we can apply the induction hypothesis. Indeed, note that the edges offered in the first round are not incident with $V'$, and hence all edges in $B'[U',V']$ and $B'[V']$, apart from the $q-1$ initial blue edges, are free. This ensures that Waiter can force a copy of $F'$ where each of the roots $r_1,\ldots,r_s,r_i'$ is mapped to its fixed image $u_1,\ldots,u_s,v$, respectively. Together with the red edge $u_iv$, we get the desired copy $\bar F$ of $F$.

\smallskip
 
\paragraph*{Case 2} $|U|\le 2$.

Then there are at most two trees in $F$, and hence $e(F)<\frac{2}{3}(t-q)$.
In this case we can remove all vertices in $V$ which are incident to a blue edge, thus obtaining a subset $V'\subset V$ of $t':=t-q$ vertices. Let $B':=B[U\cup V']$ be the new board.
The smaller board $B'$ has $s+t'$ vertices and no blue edges. 
Since $e(F)<\frac23(t-q)\le t'$ 
and each component $T_i$ in $F$ satisfies $e(T_i)<\frac{t-q}{3}=\frac{t'}{3}$, 
the assertion follows from the induction base case applied to the board~$B'$. 
\end{proof}

The last result of this subsection is an observation that if the board is much larger than the target forest, then Waiter can easily achieve her goal.  

\begin{lemma}\label{lem:doubleroot1v}
	Let $F$ be a rooted forest  
	and let $B$ be a complete graph on at least $v(F)+e(F)$ vertices.
	Then Waiter can force a red copy $\bar F$ of $F$ in $B$ 
	within at most $e(F)$ rounds in such a way that 
\begin{enumerate}[(a)]
	\item\label{doubleroot1v_roots}
     all roots are mapped into vertices fixed at the beginning, 
	\item\label{doubleroot1v_nonleaf}
     every colored edge has an endpoint in a non-leaf of $\bar F$,
	\item\label{doubleroot1v_bluestars}
     all blue edges in $E_B(V(\bar F),V(B)\setminus V(\bar F))$ 
		form a star forest such that every star is centered in a 	
		vertex $u\in V(\bar F)$ and has at most $\deg_{\bar F}(u)$ 	
		edges.
	\end{enumerate}	
\end{lemma}

\begin{proof}
Waiter has enough space to force a red copy of every component 
of $F$ in a greedy way: starting from its roots, 
Waiter always offers two edges incident to a vertex of a partial 
embedding of $F$, with both other endpoints free. Then the 
properties \ref{doubleroot1v_roots}, \ref{doubleroot1v_nonleaf} and \ref{doubleroot1v_bluestars} are satisfied.
\end{proof}

\subsection*{Forcing forests with double-rooted components}

We have dealt with rooted forests so far. In the remaining part of the section we analyze games in which Waiter's aim is to obtain a red forest with double-rooted components (trees). We abandon the forest-tuples notation while stating the theorems, but it will still be present in the proofs.

\begin{lemma}\label{lem:doubleroot0} 
	Let $F$ be a forest such that each of its components is rooted or double-rooted. Assume that every tree in $F$ has at least 1 and less than $\frac{e(F)-d}{3}$ edges, 
	where $d$ is the number of double-rooted components of $F$, 
    and that in every double-rooted component its roots 
	are at distance at least 4.
	Then playing on a complete graph $B$ on $n\ge v(F)$ vertices, Waiter can force a red copy $\bar F$ of $F$ within at most $e(F)+d+1$ rounds
	in such a way that all roots are mapped into vertices fixed at 
	the beginning.	
\end{lemma}

\begin{proof}
The main idea is to split every double-rooted tree into a forest consisting of two trees by removing one edge, and then follow the strategy from Lemma~\ref{lem:starind}. The main strategy will be interrupted at some point in order to connect such pairs of trees into one.

We call a double-rooted tree $T$ a \emph{path-double-rooted} tree, if every inner vertex of a path from one root to the other has degree exactly $2$ in $T$; in other words, the path connecting the roots in $T$ is a bare path. Suppose that $F = T_1 \dcup T_2 \dcup \ldots \dcup T_{\ell}$, where the first $d_1$ trees are double-rooted but not path-double-rooted, the next $d_2$ trees are path-double-rooted, and the remaining trees are rooted which means that they have exactly one root. Let $R_1$, $R_2$ and $R_3$ be the sets of roots of the above trees of three kinds, respectively, and let $V(B) = U_1\dcup U_2 \dcup U_3\dcup V$, where $U_i$ consists of the images of the roots in $R_i$, for $i\in[3]$. 

Consider a double-rooted tree $T_i$ which is not path-double-rooted, i.e.\ with $i\in[d_1]$. Suppose that $r_i$ and $r'_i$ are its two roots, $P_i$ is a path connecting $r_i$ to $r'_i$, and let $v$ be an inner vertex of degree at least 3 in $T_i$. Moreover, without loss of generality suppose that $v$ and $r'_i$ are at distance at least 2. 
Let $x$ be the neighbor of $v$ on a~path to $r'_i$ and let $y$ be the neighbor of $x$ on a path to $r'_i$ (note that if $v$ and $r'_i$ are at distance 2, then $y=r'_i$). 
Finally, let $w$ be a~neighbor of $v$ which does not belong to $P_i$. Set $T_i':=T[V(P_i)\cup\{w\}]$, $T_i^{(1)}\cup T_i^{(2)}:=T_i'-\{xy\}$  (see Figure~\ref{Fig:double-rooted_tree}) and $F_i':=T_i-xy$.
Furthermore, for every path-double-rooted tree $T_j$, i.e.\ with $j=d_1+k$ for $k\in[d_2]$, let $P_j$ be the path between the roots $r_j$ and $r'_j$ in $T_j$ and let $q$ be the neighbor of $r'_j$ on $P_j$. We set $T_j'=P_j-\{r'_j\}$ and $F_j':=F_j-r'_jq$. Hence $T_j'$ is just a~rooted path where the root is a~vertex of degree 1.

\begin{figure}[H]\label{normaldoublerooted}
    \begin{tikzpicture}
        \centering
        \coordinate (r1) at (1,0);
        \coordinate (v) at (4,0);
        \coordinate (x) at (6,0);
        \coordinate (y) at (8,0);
        \coordinate (r2) at (11,0);
        \coordinate (w) at (4,-2);
        \begin{pgfonlayer}{front}
    		\foreach \i in {r1,v,x,y,r2,w} \fill (\i) circle (2pt);	
            \draw (w) -- (v) -- (x) -- (y) (r1) -- (2,0) (3,0) -- (v) (y) -- (9,0) (10,0) -- (r2);
            \draw[dotted] (2.1,0) -- (2.9,0) (9.1,0) -- (9.9,0);
            \node at (1,0) [above] {$r_i$};
            \node at (4,0) [above] {$v$};
            \node at (6,0) [above] {$x$};
            \node at (8,0) [above] {$y$};
            \node at (11,0) [above] {$r'_i$};
            \node at (4,-2) [below] {$w$};
        \end{pgfonlayer}

        \draw[dashed] (1,0.5) arc[radius=0.5,start angle= 90,end angle=270];
        \draw[dashed] (3.5,-2) arc[radius=0.5,start angle= 180,end angle=360];
        \draw[dashed] (6,-0.5) arc[radius=0.5,start angle= 270,end angle=450];
        \draw[dashed] (1,-0.5) -- (3.5,-0.5) -- (3.5,-2);
        \draw[dashed] (4.5,-2) -- (4.5,-0.5) -- (6,-0.5);
        \draw[dashed] (6,0.5) -- (1,0.5);

        \draw[dashed] (8,0.5) arc[radius=0.5,start angle= 90,end angle=270];
        \draw[dashed] (11,-0.5) arc[radius=0.5,start angle= 270,end angle=450];
        \draw[dashed] (8,0.5) -- (11,0.5) (8,-0.5) -- (11,-0.5);

        \node at (2,-1) [below] {$T_i^{(1)}$};
        \node at (10,-1) [below] {$T_i^{(2)}$};

    \end{tikzpicture}
\caption{\label{Fig:double-rooted_tree} The subtree $T_i'$ of the double-rooted tree $T_i$, which is not path-double-rooted. The forest $F_i'$ contains trees $T_i^{(1)}$ and $T_i^{(2)}$, which are forced before the rest of $T_i$.}
\end{figure}
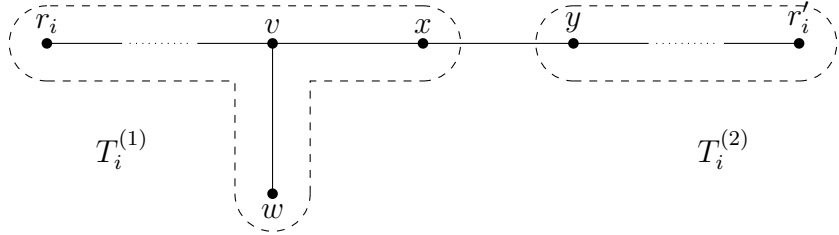

We now consider the rooted forest-tuple $F'=(F_1', F_2',\ldots,F_{\ell}')$, where for 
$h > d_1+d_2$, we set $F_h' = T_h$. We claim that $F'$ satisfies the assumptions of Lemma~\ref{lem:starind}. 
We need to check whether $F'$ is suitable. Since $e(T_i)\leq \frac{e(F)-d}{3}$ for any $i\in[\ell]$, and since we have removed from $F$ exactly $d$ edges, we have $a(F')\leq \frac{e(F)-d}{3} = \frac{e(F')}3$. Thus, by Lemma~\ref{suitable}, $F'$ is suitable.

Waiter now plays according to the strategy in the proof of  Lemma~\ref{lem:starind}, which will be interrupted in order to connect the roots in each image of a double-rooted tree as follows.
Before we proceed, let us recall that the operations of type 1 or 2 are the heart of the proof of Lemma~\ref{lem:starind}, so in view of Remark~\ref{rem:order}, Waiter has some flexibility while choosing the order of the edge embedding for a given forest.  
Hence, while forcing $F_i'$ for $i\in[d_1]$, she can first force copies of $T_i^{(1)}$ and $T_i^{(2)}$, and while forcing $F_j'$ for $j=d_1+k$ with $k\in[d_2]$, she first forces $T_j'$.

Suppose that for a tree $T_i$ which is double-rooted but not path-double-rooted, Waiter already forced in $B$ its subforest $\bar{F}_i'$ (see Figure~\ref{Fig:double-rooted_tree}) in such a way, that the only colored edges spanned by vertices in $V(\bar{F}_i')$ are the red edges in $\bar{F}_i'$. In particular, the edges $\bar{y}\bar{x}$ and $\bar{y}\bar{w}$, where $\bar{x},\bar{y},\bar{w}$ are the images of $x, y, w$, respectively, are still free. Hence, Waiter offers $\bar{y}\bar{x}$ and $\bar{y}\bar{w}$. Now no matter which of the two edges Client chooses, Waiter forces a red copy $\bar{T}_i'$ of the double-rooted subtree $T_i'\subset T_i$, possibly changing the role of $x$ and $w$. Waiter repeats this strategy for each tree $T_i$ with $i\in[d_1]$ using exactly $d_1$ rounds to create $d_1$ red edges 
not in $E\big(\bigcup_{i\in[d_1]}\bar{F}_i'\big)$ (but included in $\bar F$).

As for a path-double-rooted tree $T_j$, suppose that at some point Waiter forced a copy $\bar{F}_j'$ of a rooted path $P_j-\{r'_j\}:=r_jv_1v_2 \dots v_q$, which is a part of the double-rooted path $P_j$ with one missing edge $v_qr'_j$. Let $\bar{r}_j, \bar{v}_1, \bar{v}_2, \dots, \bar{v}_q$ be the corresponding images of vertices in $B$. Note also that since the length of $P$ is at least 4, we have $q\geq 3$.  Let $\bar{r}'_j$ be the image of $r'_j$ in $U_2$. Then all edges between $\bar{r}'_j$ and $V(\bar{F}_j')$ are free. First Waiter offers the edges $\bar{v}_q\bar{v}_{q-2}$ and $\bar{v}_q\bar{v}_{q-3}$ which, again by the hypothesis of Lemma~\ref{lem:starind}, are free. If Client chooses $\bar{v}_q\bar{v}_{q-2}$, then in the following round Waiter offers $\bar{r}'_j\bar{v}_q$ and $\bar{r}'_j\bar{v}_{q-1}$. But then, even if Client colors $\bar{r}'_j\bar{v}_q$ blue, a red path of length $v(P_j)$ between $\bar{r}_j$ and $\bar{r}'_j$ is created, namely we have a red path $\bar{r}_j\bar{v}_1\dots\bar{v}_{q-2}\bar{v}_q\bar{v}_{q-1}\bar{r}'_j$. Similarly, if in the previous round Client chooses $\bar{v}_q\bar{v}_{q-3}$, then in the following round Waiter offers $\bar{r}'_j\bar{v}_q$ and $\bar{r}'_j\bar{v}_{q-2}$. Again, even if Client colors $\bar{r}'_j\bar{v}_q$ blue, a red path $\bar{r}_j\bar{v}_1\dots\bar{v}_{q-3}\bar{v}_q\bar{v}_{q-1}\bar{v}_{q-2}\bar{r}'_j$ is created (see Figure~\ref{Fig:double-rooted_path}). Hence, Waiter forces a red copy of $\bar{P}_j$ of the double-rooted path $P_j$. Waiter repeats this strategy for each tree $T_j$, $j=d_1+k$ for $k\in[d_2]$, creating $d_2$ extra red edges (not included in $\bar{F}$). 

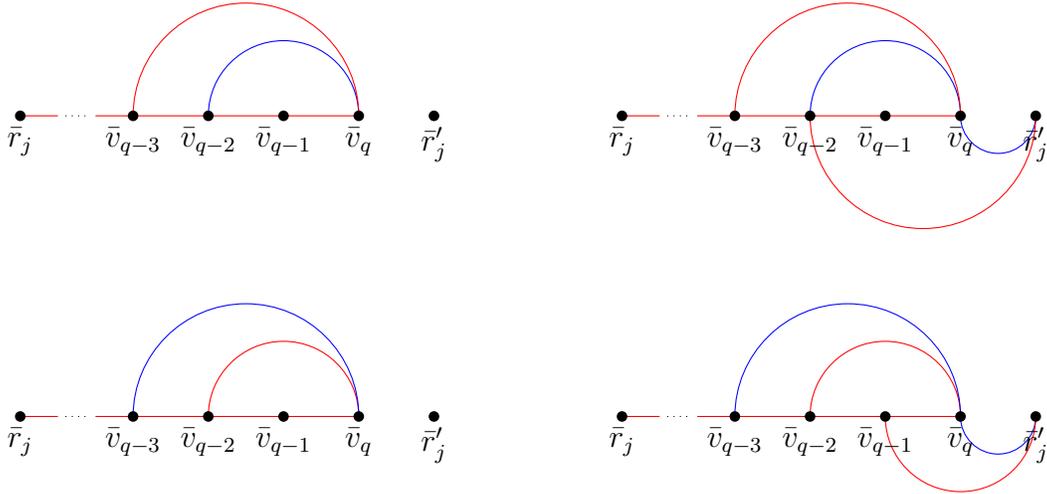
\begin{figure}[H]
    \begin{tikzpicture}
        \centering

        \coordinate (r1) at (-0.5,4);
        \coordinate (vq3) at (1,4);
        \coordinate (vq2) at (2,4);
        \coordinate (vq1) at (3,4);
        \coordinate (vq) at (4,4);
        \coordinate (r2) at (5,4);
        \begin{pgfonlayer}{front}	
            \draw[red] (r1) -- (0,4) (0.5,4) -- (vq3) -- (vq2) -- (vq1) -- (vq);
            \draw[dotted] (0.1,4) -- (0.4,4);
            \foreach \i in {r1,vq3,vq2,vq1,vq,r2} \fill (\i) circle (2pt);
            \node at (-0.5,4) [below] {$\bar{r}_j$};
            \node at (1,4) [below] {$\bar{v}_{q-3}$};
            \node at (2,4) [below] {$\bar{v}_{q-2}$};
            \node at (3,4) [below] {$\bar{v}_{q-1}$};
            \node at (4,4) [below] {$\bar{v}_q$};
            \node at (5,4) [below] {$\bar{r}'_j$};
        \end{pgfonlayer}

        \draw[blue] (4,4) arc[radius=1,start angle=0,end angle=180];
        \draw[red] (4,4) arc[radius=1.5,start angle=0,end angle=180];

        \coordinate (r1) at (7.5,4);
        \coordinate (vq3) at (9,4);
        \coordinate (vq2) at (10,4);
        \coordinate (vq1) at (11,4);
        \coordinate (vq) at (12,4);
        \coordinate (r2) at (13,4);
        \begin{pgfonlayer}{front}
            \draw[red] (r1) -- (8,4) (8.5,4) -- (vq3) -- (vq2) -- (vq1) -- (vq);
            \draw[dotted] (8.1,4) -- (8.4,4);
            \foreach \i in {r1,vq3,vq2,vq1,vq,r2} \fill (\i) circle (2pt);	
            \node at (7.5,4) [below] {$\bar{r}_j$};
            \node at (9,4) [below] {$\bar{v}_{q-3}$};
            \node at (10,4) [below] {$\bar{v}_{q-2}$};
            \node at (11,4) [below] {$\bar{v}_{q-1}$};
            \node at (12,4) [below] {$\bar{v}_q$};
            \node at (13,4) [below] {$\bar{r}'_j$};
        \end{pgfonlayer}

        \draw[blue] (12,4) arc[radius=1,start angle=0,end angle=180];
        \draw[red] (12,4) arc[radius=1.5,start angle=0,end angle=180];
        \draw[blue] (12,4) arc[radius=0.5,start angle=180,end angle=360];
        \draw[red] (10,4) arc[radius=1.5,start angle=180,end angle=360];
        
        \coordinate (r1) at (-0.5,0);
        \coordinate (vq3) at (1,0);
        \coordinate (vq2) at (2,0);
        \coordinate (vq1) at (3,0);
        \coordinate (vq) at (4,0);
        \coordinate (r2) at (5,0);
        \begin{pgfonlayer}{front}
            \draw[red] (r1) -- (0,0) (0.5,0) -- (vq3) -- (vq2) -- (vq1) -- (vq);
            \draw[dotted] (0.1,0) -- (0.4,0);
            \foreach \i in {r1,vq3,vq2,vq1,vq,r2} \fill (\i) circle (2pt);	
            \node at (-0.5,0) [below] {$\bar{r}_j$};
            \node at (1,0) [below] {$\bar{v}_{q-3}$};
            \node at (2,0) [below] {$\bar{v}_{q-2}$};
            \node at (3,0) [below] {$\bar{v}_{q-1}$};
            \node at (4,0) [below] {$\bar{v}_q$};
            \node at (5,0) [below] {$\bar{r}'_j$};
        \end{pgfonlayer}

        \draw[red] (4,0) arc[radius=1,start angle=0,end angle=180];
        \draw[blue] (4,0) arc[radius=1.5,start angle=0,end angle=180];

        \coordinate (r1) at (7.5,0);
        \coordinate (vq3) at (9,0);
        \coordinate (vq2) at (10,0);
        \coordinate (vq1) at (11,0);
        \coordinate (vq) at (12,0);
        \coordinate (r2) at (13,0);
        \begin{pgfonlayer}{front}
            \draw[red] (r1) -- (8,0) (8.5,0) -- (vq3) -- (vq2) -- (vq1) -- (vq);
            \draw[dotted] (8.1,0) -- (8.4,0);
            \foreach \i in {r1,vq3,vq2,vq1,vq,r2} \fill (\i) circle (2pt);	
            \node at (7.5,0) [below] {$\bar{r}_j$};
            \node at (9,0) [below] {$\bar{v}_{q-3}$};
            \node at (10,0) [below] {$\bar{v}_{q-2}$};
            \node at (11,0) [below] {$\bar{v}_{q-1}$};
            \node at (12,0) [below] {$\bar{v}_q$};
            \node at (13,0) [below] {$\bar{r}'_j$};
        \end{pgfonlayer}

        \draw[red] (12,0) arc[radius=1,start angle=0,end angle=180];
        \draw[blue] (12,0) arc[radius=1.5,start angle=0,end angle=180];
        \draw[blue] (12,0) arc[radius=0.5,start angle=180,end angle=360];
        \draw[red] (11,0) arc[radius=1,start angle=180,end angle=360];

    \end{tikzpicture}
\caption{\label{Fig:double-rooted_path} Connecting the roots $\bar{r}_j$ and $\bar{r}'_j$ to get a double-rooted path $\bar{P}_j$: in the first round Waiter offers edges $\bar{v}_q\bar{v}_{q-3}$ and $\bar{v}_q\bar{v}_{q-2}$.}
\end{figure}



In view of  Lemma~\ref{lem:starind}, while creating a red copy of $F'$, Waiter creates only one red edge not included later in $\bar F$. Thus the above strategy guearantees 
the desired red copy $\bar{F}$ of $F$ within at most $e(F) + 1 + d_2 \leq e(F) + d + 1$ rounds. 
\end{proof}

While most of the previous lemmas deal with forests in which the size of each component is bounded, we will also need strategies for forests where the components can be larger but the maximum degree is bounded sufficiently. 

\begin{lemma}\label{lem:doubleroot1deg}
	Let $F$ be a forest such that each of its components 
	is rooted or double-rooted.
	Suppose that $d\ge 0$ of the components 
	of $F$ are double-rooted trees,
	and the two roots of such components are distant by at least $7$.
	Let $B$ be a complete graph on at least $v(F)+\Delta(F)$ 
	vertices. Then Waiter can force a copy $\bar F$ of $F$ in $B$
	within at most $e(F)+d$ rounds in such a way that
	\begin{enumerate}[(a)]
	\item\label{doubleroot1deg_roots} 
        all roots are mapped into vertices fixed at the beginning,
	\item\label{doubleroot1deg_nonleaf}  
        every colored edge has an endpoint in a non-leaf of $\bar F$,
	\item\label{doubleroot1deg_bluestars}  
        every non-leaf of $\bar F$ is incident with at most 
		$\Delta(F)$ blue edges whose other endpoint is in 
		$V(B)\setminus V(\bar F)$.
	\end{enumerate}	 
\end{lemma}

\begin{proof}
Let $T_1,T_2,\ldots,T_d$ be the components of $F$ with two roots,
and denote with $x_i$ and $y_i$ the roots of $T_i$ for every 
$i\in [d]$. Let $x_i'$ and $y_i'$ be their 
fixed images, and
let $P_i$ be the unique path between $x_i$ and $y_i$ in $F$, where we let $d_i=e(P_i)\geq 7$.
Before the game starts, partition $V(B)$ into vertex sets $V_1,V_2,\ldots,V_d,V_{d+1}$ such that
for every $i\in [d]$, we have
$|V_i|=d_i+1$ and that $V_i$ 
contains $x_i'$ and $y_i'$ but no other fixed image of roots in $F$. 
Let $F_0$ be the union of all paths $P_i$.

In the first stage of the game, for each $i\in [d]$, Waiter plays on each of the boards $B[V_i]$ according to
the strategy from Lemma~\ref{lem:Ham.path.ends.fixed}
in order to claim a Hamilton path in $B[V_i]$
between $x_i'$ and $y_i'$, i.e.~a copy of $P_i$. 
In total, this stage lasts
$\sum_{i\in [d]} (e(T_i)+1) = e(F_0)+d$ rounds.

Afterwards, for the second stage, Waiter can extend
the copy of $F_0$
in a greedy way to a copy of $F$, while keeping the images of all roots fixed. To be more precise,
as long as there is a vertex $x$ in the image of a partial embedding of $F$ at which still some number $t \leq \Delta(F)$ of  edges need to be claimed,
Waiter simply plays $t$ rounds in each of which she offers two edges between $x$ and the free vertices. Note that this is possible
since $v(B)\geq v(F)+\Delta(F)$.
Moreover, this second stage lasts precisely $e(F_0)-e(F_1)$ rounds,
so that after a total of $e(F)+d$ rounds
Waiter has forced a copy of $F$. Property~\ref{doubleroot1deg_roots} is immediate. Property~\ref{doubleroot1deg_nonleaf} holds since in the
first stage no leaf is embedded and since in the second stage, Waiter only offers edges incident to some vertex $x$ which already belongs to the embedding and still needs at least one more neighbor, thus cannot be a leaf. Moreover, as Waiter only offers edges between such a vertex $x$ and the remaining free vertices
as long as $x$ is still missing some neighbors,
which happens at most $\Delta(F)$ rounds, property~\ref{doubleroot1deg_bluestars}  follows.
\end{proof}

The following lemma 
is a variant of Theorem 5.2 in~\cite{clemens2020fast} and
can be proven analogously. It roughly states that a spanning tree of not too large maximum degree can be forced by Waiter fast, even if the images of $r\in \{1,2\}$ vertices are fixed at the beginning of the game. In fact, the case $r=1$ is already covered by Theorem 5.2 in~\cite{clemens2020fast}, while the case $r=2$ only requires a few modifications. We therefore omit the details here, and move the proof to the appendix.

\begin{lemma}\label{lem:doubleroot2}
	There exists $\alpha>0$ and $n_\alpha$ such that 
	for every $n\ge n_\alpha$ the following holds. 
	Let $T$ be tree on $n$ vertices, with $r\in \{1,2\}$ roots and with 
	$\Delta(T)\le \alpha \sqrt n$,
	and such that its roots (if there are two of them) are distant 
	by at least 7. Then Waiter can force a copy $\bar T$ of $T$ 
	in $K_n$ within at most $e(T)+r$ rounds,
	in such a way that all roots are mapped into vertices 
	fixed at the beginning. 
\end{lemma}

Finally, we can conclude the following theorem.

\begin{theorem}\label{thm:doubleroot3}
	There exists $\alpha>0$ and $n_\alpha$ such that 
	for every $n\ge n_\alpha$ the following holds. 
	Let $F$ be a forest on $n$ vertices, with 
	$\Delta(F)\le  \alpha \sqrt n$ and
	such that every component is non-trivial and 
	rooted or double-rooted.
	Suppose that $d\geq 0$ of the components are double-rooted trees
	and for every double-rooted component its roots are 
	distant by at least 7.
	Then Waiter can force a copy $\bar F$ of $F$ 
	in $K_n$ within at most $e(F)+d+1$ rounds,
	in such a way that all roots are mapped into vertices 
	fixed at the beginning. 
\end{theorem}

\begin{proof}
Let $\alpha'$ be the constant promised by Lemma~\ref{lem:doubleroot2},
and set $\alpha = \frac{\alpha'}{3}$. 
Whenever needed, assume that $n$ is large enough.
Note that $d\leq \frac{n}{8}$, since every double-rooted component 
is required to have at least 8 vertices. 
If every component of $F$ has less than $\frac{e(F)-d}{3}$ edges, 
then the assertion follows from Lemma \ref{lem:doubleroot0}.
Hence, assume from now on that there is a component $T$ in $F$ 
which has at least $\frac{e(F)-d}{3}\geq \frac{n}{8}$ edges 
and put $F'=F - V(T)$. 
Let $R\subseteq V(T)$ be the set of roots of $T$, and note that
$$
n-|R|\geq v(F')+ v(T)-2\ge v(F')+ \frac{n}{8}-2\ge v(F')+ \Delta(F')
$$
for $n$ big enough. Let $\bar R$ be the vertices in $K_n$ 
that the roots in $R$ need to be mapped to.
Lemma~\ref{lem:doubleroot1deg} implies that Waiter can force 
a copy $\bar{F'}$ of $F'$ in $K_n-{\bar R}$
such that the properties \ref{doubleroot1deg_roots} and \ref{doubleroot1deg_nonleaf} 
from this lemma hold. 
In particular, all roots of $F'$ are mapped to the vertices fixed at 
the beginning, and no edge in in $V(K_n)\setminus V(\bar{F'})$ is 
colored yet. Moreover, it takes her at most $e(F')+d'$ rounds, where 
$d'\in \{d-1,d\}$ is the number of double-rooted trees of $F'$.
Afterwards, Waiter plays on the board $K_n- V(\bar F')$.
In view of Lemma~\ref{lem:doubleroot2} she can force a copy of $T$, 
within at most $e(T)+|R|$ rounds such that the vertices of $R$ are 
mapped to the fixed vertices in $\bar R$.
For this, note that 
$\Delta(T)\leq \alpha \sqrt{n} < \alpha' \sqrt{v(T)}$.
In particular, Waiter forces a copy of $F$ as required within
$e(F')+d'+e(T)+|R|=e(F)+d+1$ rounds.
\end{proof}

\medskip

\section{Forcing spanning trees with linear maximum degree}
\label{sec:proof.linear.degree}

In this section we prove Theorem \ref{thm:giventree}.

Let $\eps\in \left(0,\frac{1}{3}\right)$ be given, 
set $\delta=\frac{\eps}{3}$,
let $\alpha$ be the constant from Theorem \ref{thm:doubleroot3},
and let $b=20\eps^{-1}\alpha^{-1}$.
Whenever needed assume that $n$ is large enough.
Let $T$ be a tree on $n$ vertices with $\Delta(T)\leq \left(\frac{1}{3}-\eps\right)n$. 
We set $D=\frac{\eps\alpha \sqrt n}{5}$, 
and we say that a vertex of $T$ is \emph{big} if its degree in $T$ 
is at least $D$. We denote the set of all big 
vertices by $\BigT$, and 
we let $T'$ be the smallest subtree of $T$ 
containing all vertices from $\BigT$. 
We say that a bare path in $T'$ is \emph{special}
if each of its inner vertices does not belong to $\BigT$.
Let $\mathcal{P}$ be the family of all maximal special bare paths of 
$T'$ whose length is at least $11$,
and let $\mathcal{P}'$ be the family of all the paths of $\mathcal{P}$ without 
their endpoints. Note that then the paths in $\mathcal{P}'$ are 
vertex-disjoint and have length at least $9$.
With $I'$ we denote the set of all inner vertices of the paths in $\mathcal{P}'$. 
Moreover, we set
$$
F_1:=
\begin{cases}
T',~ & \text{if }v(T')\leq \frac{\eps n}{4}, \\
T'-I',~ & \text{if }v(T')> \frac{\eps n}{4}.
\end{cases}
$$

In the following, we first give a brief overview of
Waiter's strategy for a game played on $B:=K_n=(V,E)$, and afterwards we provide more details on why Waiter can play as suggested and force a copy of $T$ 
as required.

\subsection*{Strategy description.} 

Waiter's strategy consists of three stages. Details of each stage will be given in the strategy discussion. 

\paragraph{\textbf{Stage I}}~\\
By an application of Lemma~\ref{lem:doubleroot1v},
Waiter forces a red copy $\bar{F_1}$ of $F_1$ 
within $e(F_1)$ rounds such that 
immediately afterwards the following holds:
\begin{enumerate}[label=(I.\arabic*)]
	\item\label{stage1endpoint}
        every colored edge has an endpoint in $V(\bar{F_1})$.
	\item\label{stage1bluestars} 
        all blue edges in 
		$E_{B}(V(\bar{F_1}),V\setminus V(\bar{F_1}))$ 
		form a star forest such that every star is centered at 
		a~vertex $u\in V(\bar{F_1})$ and has at most 
		$\deg_{\bar{F_1}}(u)$ edges.
\end{enumerate}
Next, let $f_1:V(F_1)\rightarrow V(\bar{F_1})$ 
denote the embedding of $F_1$ into the red graph.
Then Waiter proceeds with Stage II.

\paragraph{\textbf{Stage II}}~\\ 
Let $C_1,\ldots,C_s$ be the non-trivial components of
$F'_2:=T\setminus E(F_1)$. 
For each $i\in [s]$, let $R_i$ denote the set of vertices 
of $C_i$ that belong to $V(F_1)$, and call 
these vertices the roots of $C_i$.
Note that $1\leq |R_i|\leq 2$ for every $i\in [s]$.
Furthermore, $C_i$ has two roots $r_i,r_i'$ if and only if  
these vertices are connected in $C_i$ 
by a path from $\mathcal{P}'$.
Let $R:=\bigcup_{i\in [s]} R_i$, $B_2:=B[(V\setminus V(\bar{F_1}))\cup f_1(R)]$
and $t:=n-v(F_1)$.
We distinguish two cases.

\smallskip
\paragraph*{Case 1} 
If $F_1=T'$ and all components $C_i$ have size 
at most $\left(\frac{1}{3}-\delta\right)t$, then Waiter finishes the 
embedding of $T$ within $e(F'_2)+1$ further rounds, 
by an application of Theorem~\ref{thm:root1del} on the board $B_2$, and she skips
Stage~III.

\smallskip
\paragraph*{Case 2}
Otherwise, let $F_2\subset F'_2$ be the star forest consisting of all
edges incident with $R$ in $F'_2$. (This includes the first and last 
edge of each path in $\mathcal{P}'$.) Let us assume that the vertices in $R$ are the roots of $F_2$. 
Then, by an application of Theorem~\ref{thm:root1del},
playing on $B_2$
for $e(F_2)$ rounds, 
Waiter forces a red copy $\bar{F_2}$ of  $F_2$
such that $f_1(R)$ is the set of roots of $\bar{F_2}$ and the following holds.
\begin{enumerate}[label=(II.\arabic*)]
	\item\label{stage2roots} 
        For each $r\in R$, the star with center 
		$\bar r:=f_1(r)$ has size $\deg_{F_2}(r)$.
	\item\label{stage2colored} 
        All colored edges from Stage II intersect 
		$V(\bar{F_2})$.
	\item\label{stage2nonleaf} 
        Every vertex in 
		$V\setminus (V(\bar{F_1})\cup V(\bar{F_2}))$ 
		is incident with at most one colored edge
		and the other endpoint of such an edge is not a leaf of 
		$\bar{F_2}$.
\end{enumerate}
Afterwards, let
$f_2:V(F_1\cup F_2) \rightarrow V(\bar{F_1}\cup \bar{F_2})$ 
denote the obtained embedding (an extension of $f_1$) of $F_1\cup F_2$ into Client's graph.
Then, Waiter proceeds with Stage III.

\paragraph{\textbf{Stage III}}~\\  
When Waiter enters this stage,
Client's graph contains a copy of $F_1\cup F_2$.
Let $F_3:=T\setminus (E(F_1)\cup E(F_2))$.
Then, by an application of Theorem~\ref{thm:doubleroot3}, 
within at most $e(F_3)+\frac{4n}{D}+1$ rounds, Waiter finishes a~copy of $T$. 

\subsection*{Strategy discussion} 
We now explain, separately for each stage, why Waiter is able to play as described above.
The following properties of $\mathcal{P}'$, $I'$ and $F_1$ will be useful in the further analysis.

\begin{claim}\label{claim:F1}
For large enough $n$ we have the following.
 \begin{enumerate}[(a)]
 \item\label{F1paths}
 $|\mathcal{P}'|< \frac{4n}{D}$.
 \item\label{F1inner}
 If $v(T')> \frac{\eps n}{4}$, then $|I'|>v(T')-\frac{52n}{D}$.
 \item\label{F1v}
 $v(F_1)\leq \frac{\eps n}{4}$.
 \end{enumerate}
\end{claim}

\begin{proof}
Note that by the definition of $\BigT$ we have $|\BigT|< \frac{2n}{D}$. 
Clearly $\Delta(T')\le |L(T')|< \frac{2n}{D}$,
since every leaf of $T'$ is a big vertex of $T$. 
Define a set 
$\GoodT=\{v\in V(T')\setminus \BigT:~\deg_{T'}(v)=2\}$.
Next, observe that the number of vertices of $T'$ 
which belong to $\BigT$ or have degree at least 3 in $T'$ 
is less than $|\BigT|+|L(T')| < \frac{4n}{D}$. Hence,
$|\GoodT|> v(T') - \frac{4n}{D}$. 
Since every endpoint of a maximal special bare path is either 
a big vertex (including all leaves of $T'$) or 
it has degree more than 2 in $T'$, 
the number of maximal special bare paths in $T'$ is less than 
$\frac{4n}{D}$. Thus $|\mathcal{P}'|\leq |\mathcal{P}|< \frac{4n}{D}$,
and the first part of the claim follows.

For the second part, we assume that $v(T')>\frac{\eps n}{4}$ and 
we denote by $I$ the set of all inner vertices of the paths in $\mathcal{P}$.
We already verified that $|\mathcal{P}|<\frac{4n}{D}$, hence 
the number of inner vertices of maximal special bare 
paths of length at most 10 is less than $\frac{40n}{D}$. 
Therefore 
$$
|I'|=|I|-2|\mathcal{P}| 
	> |\GoodT|-\frac{40n}{D} - \frac{8n}{D} 
	> v(T')-\frac{52n}{D}.
$$
In the last inequality we used the fact $|\GoodT|> v(T') - \frac{4n}{D}$ proved above.

In order to prove the third part of the claim, we assume that 
$v(T')>\frac{\eps n}{4}$ and $F_1=T'-I'$, since otherwise $v(F_1)=v(T')\le \frac{\eps n}{4}$.
Then, based on (b), we infer that $v(F_1)= v(T') - |I'| < \frac{52n}{D} < \frac{\eps n}{4}$
for large enough $n$.
\end{proof}

\paragraph{\textbf{Stage I analysis}} 
In view of Claim~\ref{claim:F1}\ref{F1v}, we know 
that $v(F_1)+e(F_1)<\frac{\eps n}{2}< n$. Thus,
by Lemma~\ref{lem:doubleroot1v} (with $B:=K_n$),
Waiter can force a red copy of $F_1$ as described.

\paragraph{\textbf{Stage II analysis}} 
When Waiter enters Stage II, 
she has already forced a red copy $\bar{F_1}$ of $F_1$, 
and the roots of the non-trivial components $C_i$ of $F'_2$ have fixed images $f_1(R)$ on the board $B_2$. 

We start with the discussion of Case 1. 
For each component $C_i$ we have
$e(C_i)\leq \left(\frac{1}{3}-\delta \right)t$
and $|R_i|=1$, i.e.~there is exactly one root $r_i$, 
since $F_1=T'$. Now, we 
put $F_2=\bigcup_{i\in[s]}C_i$
and consider the game played on $B_2$, which is a complete graph 
on $t+s$ vertices, with some edges colored. Note that $t=e(F_2)$ and the properties 
\ref{stage1endpoint} and \ref{stage1bluestars} imply that the blue edges in $B_2$
form a~star forest with $f_1(R)$ being the set of centers of the stars.
Moreover, the number of these edges is bounded by the 
number of previous rounds, which is $e(F_1)\leq \frac{\eps n}{4}<\delta t$.
Thus, by Theorem~\ref{thm:root1del} with $q=0$, Waiter can force a red
copy $\bar{F_2}$ of $F_2$ in $B_2$ within $e(F_2)+1$ rounds and such that
all vertices in $R$ are mapped to their fixed images under $f_1$. 
Then, the union of $\bar{F_1}$ and $\bar{F_2}$ is a copy of $T$ 
obtained within a total of $e(F_1)+e(F_2)+1=e(T)+1$ rounds.

Next, we consider Case 2. First we will verify that 
\begin{equation}\label{bigboard}
v(B_2)-v(F_2)>\frac{\eps n}{5}.
\end{equation}

If $F_1\neq T'$, then $v(T')>\frac{\eps n}{4}$, so by Claim~\ref{claim:F1}\ref{F1v} we get
$$
|I'|>v(T')-\frac{52n}{D}>\frac{\eps n}{4}-\frac{52n}{D}>\frac{\eps n}{5}
$$ 
for large enough $n$. Thus, in this case (\ref{bigboard}) holds, since $v(B_2)-v(F_2)\ge |I'|$. 

If otherwise $F_1=T'$, but there is a component $C_i$ such that
$e(C_i)>\left(\frac{1}{3}-\delta\right)t$, then for the unique root 
$r_i$ of $C_i$ we have
$$
e(C_i)-\deg_{F_2}(r_i) 
	>\left(\frac{1}{3}-\delta\right)t-\left(\frac{1}{3}-\eps\right)n
    =\left(\frac{1}{3}-\delta\right)(n-v(F_1))-\left(\frac{1}{3}-\eps\right)n
	>\frac{\eps n}{5},
$$
where in the last inequality we put $\delta:=\frac{\eps}{3}$ and use Claim~\ref{claim:F1}\ref{F1v}. 
Since $v(B_2)-v(F_2)\ge e(C_i)-\deg_{F_2}(r_i)$, we obtain (\ref{bigboard}) again.

Observe that each of the components of the rooted star forest $F_2$ has size at most 
$\Delta(T)\le  \left(\frac{1}{3} - \eps\right)n < \left(\frac{1}{3} - \delta\right)t$,
and has only one root. 
Therefore, by the analogous argument as in Case 1, using Theorem~\ref{thm:root1del}, 
Waiter can force a red copy $\bar{F_2}$ of $F_2$
on $B_2$ such that all roots are mapped
to their fixed images under $f_1$.
As $v(B_2)>v(F_2)$, Waiter needs only $e(F_2)$ rounds of Stage II.
The properties \ref{stage2roots} -- \ref{stage2nonleaf} follow from
the properties \ref{root1del_roots} -- \ref{root1del_nonleaf} in Theorem~\ref{thm:root1del}.

\paragraph{\textbf{Stage III analysis}} 
It follows from (\ref{bigboard}) and the definitions of $F_3$ and $B_2$, 
that $e(F_3)>\frac{\eps n}{5}$.
Let $C_1',\ldots,C_{s'}'$ be the non-trivial components of $F_3$.
In view of the description of Stage I and Stage II, for every vertex $u\in V(F_1)$ we have
$\deg_{\bar F_1\cup \bar F_2}(f_2(u))=\deg_T(u)$, so there are no vertices of $F_1$ in 
$\bigcup_{i\in[s']}C_i'$.

For each $i\in [s']$ let $R_i'$ denote the set of vertices 
of $C_i'$ that belong to $V(F_2)$, and call 
these vertices the roots of $C_i'$.
Note that $1\leq |R_i'|\leq 2$ for every $i\in [s]$ and,
if $C_i'$ has two roots, then 
these vertices are connected in $C_i'$ 
by a special bare path of length at least $7$ 
(i.e.~a path from $\mathcal{P}'$ minus its endpoints).
Let $R':=\bigcup_{i\in [s]} R_i'$.
Since all big vertices were embedded during Stage I,
and since $e(F_3)>\frac{\eps n}{5}$,
we know that every vertex in $F_3$ has degree at most 
$D=\frac{\eps\alpha \sqrt n}{5}<\alpha \sqrt{v(F_3)}$.
Let $d$ be the number of double-rooted components of $F_3$.
Then $d\leq \frac{4n}{D}$, 
since the number of paths in ${\mathcal P}'$ 
is at most $\frac{4n}{D}$, as we know from Claim~\ref{claim:F1}\ref{F1paths}. 

Consider the board  
$B_3:=B[(V\setminus (V(\bar{F_1}\cup \bar{F_2})) \cup f_2(R')]$.
Because of \ref{stage2nonleaf}, the vertices in $R'$ 
are leaves of $F_2$, so by \ref{stage2colored} we know that all edges in $E(B_3)$ 
are still free at the beginning of Stage III 
(except, perhaps, irrelevant edges between vertices in $f_2(R')$). 
Therefore we can apply Theorem \ref{thm:doubleroot3}, 
provided $n$ is big enough.
Hence Waiter can force a red copy $\bar{F_3}$ of $F_3$ on $B_3$
within at most $e(F_3)+d+1$ rounds such that each
vertex in $R'$ is mapped to its image under $f_2$.

Summing up, the union of $\bar{F_1}$, $\bar{F_2}$ and $\bar{F_3}$
is a copy of $T$, and Waiter played at most
$$
e(F_1)+e(F_2)+e(F_3)+d+1\le n+\frac{4n}{D} \leq n + b\sqrt n
$$ 
rounds. 
\hfill $\Box$


\section{Avoiding trees with linear maximum degree}
\label{sec:avoiding}

In the following we prove Theorem~\ref{thm:upperbound.wc}. Indeed the statement immediately follows from the result below, by considering e.g.
a spanning tree consisting of two large stars connected by a bare path.

\begin{theorem}
Let $n$ be a large enough integer.
In a Waiter-Client game on $K_n$, Client has a strategy
to avoid two disjoint red stars of size at least $0.499n$ each.
\end{theorem}

\begin{proof}
We will describe a randomized strategy for Client
and we will show that, against any strategy of Waiter,
this strategy asymptotically almost surely
prevents Client from occupying two disjoint stars of size $0.499n$.
Note that this is enough to prove the statement above,
as it implies that Waiter does not have a strategy 
to always force such stars, and hence 
Client has a deterministic strategy as desired.

We begin by ordering the vertices $\{v_1, v_2, \dots, v_n\}$ of the board $K_n$ arbitrarily, 
say, by their indices. 
In each round Waiter offers two free edges. For $p=0{.}4$, Client's strategy is as follows:
\begin{itemize}
    \item If Waiter offers a \emph{cherry} $v_iv_kv_j$, that is edges $v_iv_k$ and $v_jv_k$ with $i < j$, then Client colors red the edge $v_iv_k$ with probability $p$, and the other edge -- with probability $q=1-p$.
    \item If Waiter offers a pair of disjoint edges, that is edges $v_iv_j$ and $v_kv_{\ell}$ with $\{i,j\}\cap\{k,\ell\}=\emptyset$, then Client colors red any of them with equal probability.
\end{itemize}

Notice that since $p < 1/2$, in the first case Client prefers ``elder'' vertices, and intuitively larger red stars should appear at vertices with higher indices.
Our task is to show that asymptotically almost surely at the end of the game there are no two disjoint red stars of size at least $cn$ each, where we set $c:=0{.}499$. 

In order to analyze large disjoint stars in $C$, 
we will focus first on two vertices, say $x$ and $y$. Without loss of generality we may assume that $x < y$ in the ordering considered at the beginning of this proof. 
Observe that in order to get two disjoint stars in $C$ of size at least $cn$ each, and centered at $x$ and $y$, we need to have
\begin{align}\label{eq:2_stars_condition}
    \begin{cases}
        \deg_C(x) \geq cn, \\
        \deg_C(y) \geq cn, \\
        \deg_C(x) + \deg_C(y) - |N_C(x) \cap N_C(y)| \geq 2cn.
    \end{cases}
\end{align}

Note that at the end of the game we have 
$$
\deg_C(x) + \deg_C(y) - |N_C(x) \cap N_C(y)| + |N_W(x) \cap N_W(y)| = n-2,
$$ 
and for any vertex $v$ we have $\deg_C(v)+\deg_W(v)=n-1$. Hence, 
a condition necessary for (\ref{eq:2_stars_condition}) is
\begin{align}\label{eq:waiters_graph_condition}
    \begin{cases}
        \deg_W(x) < (1-c)n, \\
        |N_W(x) \cap N_W(y)| < (1-2c)n.
    \end{cases}
\end{align}
at the end of the game.

We will show that (\ref{eq:waiters_graph_condition}) holds with probability $o(n^{-2})$, so that later a union bound over all pairs $x,y$ will finish the argument. 
To this end we will introduce three random processes (by defining certain sets $V_1,V_2,V_3$), partially describing the course of the game, which enable us to show that a.a.s.~the blue degree $\deg_W(x)$ or the common blue neighborhood $N_W(x) \cap N_W(y)$ is larger than what (\ref{eq:waiters_graph_condition}) indicates. Before we proceed, we introduce a couple of additional definitions. During the game, a vertex $v\notin\{x,y\}$ is called:
\begin{itemize}
    \item \emph{inactive}, if both edges $xv, yv$ are free;
    \item \emph{active}, if exactly one of the edges $xv, yv$ has been colored so far.
\end{itemize} 

In order to analyze Client's randomized strategy,
we consider vertex sets $V_1,V_2,V_3$
while the game is proceeding, where initially all these sets are empty.
We update these sets after each move, depending on the following four types of moves performed by Waiter. 

\begin{itemize}
    \item[(1)] Suppose that in some round Waiter offers
    a cherry $xvy$ (and hence $v$ was inactive before this round), then we put $v$
    into $V_1$.
    
    \noindent    
    For the $i$-th move of this type, let $X_i$ denote a random variable which indicates whether Client has colored the edge $yv$ red or blue, that is
    \begin{align*}
        X_i = \begin{cases}
            1, & yv \text{ is red } (xv \text{ is blue});\\
            0, & yv \text{ is blue } (xv \text{ is red}).
        \end{cases}
    \end{align*}
    Note that the random variables $X_i$ are independent and $\Prob(X_i=1) = q$ for every $i=1,2,\dots$. Let $T_1=|V_1|$ at the end of the game. Then $T_1\leq n-2$ and $T_1$ is a random variable which may depend on Client's moves, that is $T_1$ may depend on random variables $X_1, X_2, \dots$. Furthermore, and at the end of the game 
    \begin{equation}
        \deg_W(x) \geq \sum_{i=1}^{T_1} X_i.
    \end{equation}

    \item[(2)] Suppose that in some round Waiter offers two edges which do not form a cherry as in (1), and such that one of them is an edge between $\{x,y\}$ and some inactive vertex $v$, and the other is some edge $uw$. 
We then add $v$ to $V_2$ with the following restrictions.    
If Waiter offers two edges between $\{x,y\}$ and distinct vertices $z,z'$, where both $z$ and $z'$ are inactive vertices, then we let $v$ be any of these vertices, say $z$, and we let $uw$ be the other edge, say incident with $z'$. 
In this case we add $v=z$ to $V_2$, but we do not add $z'$ to $V_2$.
Moreover, if Waiter offers two edges between $\{x,y\}$ and distinct vertices $z,z'$, such that exactly one of $z$ and $z'$ is inactive, and the other is an active vertex in $V_2$, we treat it as a move of type 3 described below, i.e.~we do not add any vertex to $V_2$. 
 
\noindent   
Having $v$ fixed for a move of type 2, we let $Y_v$ denote a random variable which indicates whether Client has colored the other edge $uw$ red or not, that is
    \begin{align*} 
        Y_v = \begin{cases}
            1, & uw  \text{ is red};\\
            0, & uw \text{ is blue}.
        \end{cases} 
    \end{align*}
    This time the random variables $(Y_v)_{v\in V_2}$ do not have to be independent, but we know that for every $v\in V_2$ we have
    \begin{align*}
        \Prob(Y_v=1) \geq p.
    \end{align*}

    \item[(3)] Suppose that in some round, Waiter offers two edges 
	neither of type 1 nor of type 2, and such that one of them is an edge 
	between $\{x,y\}$ and some active vertex $v\in V_2$, 
	and the other is some edge $uw$. 
	Then we add $v$ to $V_3$ with the following restriction.    
	If Waiter offers two edges between $\{x,y\}$ and distinct vertices $z,z'$, where 
	both $z$ and $z'$ are active vertices in $V_2$, we then let $v$ be any 
	of these vertices, say $z$, and we let $uw$ be the other edge, say 
	incident 
	with $z'$. In this case we add $v=z$ to $V_3$, but we do not add $z'$ to 
	$V_3$.
    
    \noindent
    Having $v$ fixed for a move of type 3, we let $Y'_v$ denote a random variable which indicates whether Client has colored the other edge $uw$ red or not, that is
    \begin{align*} 
        Y'_v = \begin{cases}
            1, & uw  \text{ is red};\\
            0, & uw \text{ is blue}.
        \end{cases}
    \end{align*}
    Again, the random variables $(Y'_v)_{v
    \in V_3}$ do not have to be independent, but for every $v\in V_3$ we have
    \begin{align*}
        \Prob(Y'_v=1) \geq p .
    \end{align*}
    Moreover, the event $\{Y_v=1,Y'_v=1\}$ implies that $xv, yv$ are blue or, equivalently, $v\in N_W(x)\cap N_W(y)$. 
    
	\item[(4)] Suppose that Waiter makes a move which is not of type 1, 2 or 3, then we do not update any  of the sets $V_1,V_2,V_3$.    
    
\end{itemize}

We can now move on to the final position analysis. Let $T_2 = |V_3|$ at the end of the game. We claim that 
\begin{align}\label{eq:T1+4T2}
    T_1 + 4T_2 \geq n-2.
\end{align}
Indeed, among $n-2-|T_1|$ vertices $v$ which have not been offered in a move of type 1, at most half were ignored in moves of type 2 and not included in $V_2$ (in case were Waiter offered two edges between $\{x,y\}$ and distinct vertices $z,z'$ where both $z$ and $z'$ where inactive) and at most half of the remaining ones were ignored in moves of type 3 and not included in $V_3$ (in case where Waiter offered two edges between $\{x,y\}$ 
and distinct vertices $z,z'$ where both $z$ and $z'$ were active vertices of $V_2$, and also in case where one of them was active and the other not).

Next, we define the following coupling. For $v\in V_3$, let $D_v$ be a random variable such that
\begin{align*}
    \{Y_v Y'_v = 0\} \quad \Longrightarrow \quad \{D_v = 0\},
\end{align*}
and with probability distribution given by
\begin{align*}
    D_v = \begin{cases}
        1, & \text{with probability } p^2,\\
        0, & \text{with probability } 1-p^2.
    \end{cases}
\end{align*} 
Hence
\begin{align}
    |N_W(x)\cap N_W(y)| \geq \sum_{v\in V_3} Y_vY'_{v} \geq \sum_{v\in V_3} D_v,
\end{align}
the random variables $(D_v)_{v\in V_3}$ are independent, and they have the same probability distribution. Our task now is to estimate the probability that the events in (\ref{eq:waiters_graph_condition}) hold simultaneously. By a~slight abuse of notation, i.e.~changing the indexing of random variables $(D_v)_{v\in V_3}$ to $(D_j)_{j=1}^{T_2}$, we can bound this probability from above by
\begin{align}\label{eq:prob_of_2_large_stars}
    \Prob\left(\deg_W(x)\leq (1-c)n,\, |N_W(x)\cap N_W(y)| \leq (1-2c)n\right) & \leq \Prob\left(\sum_{i=1}^{T_1} X_i\leq (1-c)n,\, \sum_{j=1}^{T_2} D_j \leq (1-2c)n\right). 
\end{align}
Let $\calE_1$ and $\calE_2$ denote the events $\sum_{i=1}^{T_1} X_i\leq (1-c)n$ and $\sum_{j=1}^{T_2} D_j \leq (1-2c)n$, respectively. 

Next, let us define the events
\begin{align*}
    \calF_1: \sum_{i=1}^{T_1} X_i \geq T_1 q - 2\sqrt{n \log n} \quad \text{ and } \quad
    \calF_2: \sum_{j=1}^{T_2} D_j \geq T_2 p^2 - 2\sqrt{n \log n}.
\end{align*}
Suppose that the events $\calE_1, \calE_2$ and $\calF_1, \calF_2$ occur simultaneously. Then we get that
\begin{align*}
    T_1 q - 2\sqrt{n \log n} \leq (1-c)n \quad \text{ and } \quad T_2 p^2 - 2\sqrt{n \log n} \leq (1-2c)n,
\end{align*}
which in turn implies that for some positive constant $\delta<0{.}0001$ and $n$ sufficiently large we have
\begin{align*}
    T_1 + 4T_2 & \leq \frac{(1-c+\delta)n}{q} + \frac{4(1-2c+\delta)n}{p^2}.
\end{align*}
Recall that $p=0{.}4$, $q=0.{6}$ and $c=0{.}499$. Hence we get
\begin{align*}
    T_1 + 4T_2 & \leq \frac{(1-c+\delta)n}{q} + \frac{4(1-2c+\delta)n}{p^2} < 0{.}89 n.
\end{align*}
This contradicts (\ref{eq:T1+4T2}). Therefore, if the events $\calE_1$ and $\calE_2$ hold, then one of the events $\calF_1, \calF_2$ cannot hold. By Lemma~\ref{lemma:Chernoff_variant}, this in turn happens with probability $o(n^{-2})$. Hence, the probability in (\ref{eq:prob_of_2_large_stars}) is bounded from above by $o(n^{-2})$. Finally, taking the union bound over all possible pairs $x, y$, we get that the probability that Client's graph contains two disjoint stars of size at least $cn$ each is $o(1)$.
\end{proof}

\medskip


\section{Client-Waiter spanning tree game}\label{sec:cw_trees}

\subsection{Auxiliary games}
\label{sec:strategy.auxiliary}
The main goal of this subsection is to prove that Waiter can force 
Client to build a graph in which every small vertex set has a large
common neighborhood. We start with the following lemma.

\begin{lemma}\label{lemma:games.on.hypergraph}
	Let $n$ be large enough.
	Let $\mathcal{H} = (X,\mathcal{F})$ be a hypergraph
	with $|X|=n$, $|\mathcal{F}|\leq n^{\log(n)}$ and such that
	every hyperedge contains at least $\log^3(n)$ vertices.
	Then in the Waiter-Client game on $\mathcal{H}$,
	Waiter has a strategy such that at the end of the game, 
	the set $C$ of Client's elements satisfies
	$$
	|C\cap f| \geq \frac{|f|}{100} \quad \text{for every }f\in \mathcal{F} .
	$$
\end{lemma}

\begin{proof}
Consider the family
$$
\mathcal{F}' :=
	\left\{ f\setminus A:~ f\in \mathcal{F},~ A\subseteq f,~ 
		|A|=\left\lfloor \frac{|f|}{100} \right\rfloor \right\}.
$$
If Client claims an element in each set of $\mathcal{F}'$, 
then we are done. To verify that Waiter can force this, 
we only need to check that the condition
from Theorem~\ref{thm:WC_transversal} holds:
\begin{align*}
\sum_{F\in \mathcal{F}'} 2^{-|F|}
	& \leq \sum_{k\in \mathbb{N} \atop k\geq \log^3(n)} 
		\sum_{f\in \mathcal{F}: \atop |f|=k}
		\sum_{A\subset f: \atop |A|
			=\left\lfloor \frac{|f|}{100} \right\rfloor}
		2^{-0.99k} 
		\leq \sum_{k\in \mathbb{N} \atop k\geq \log^3(n)} 
		\sum_{f\in \mathcal{F}: \atop |f|=k}
		\binom{k}{0.01k}
		2^{-0.99k} \\
		& \leq \sum_{k\in \mathbb{N} \atop k\geq \log^3(n)} 
		\sum_{f\in \mathcal{F}: \atop |f|=k}
		(100e)^{0.01k}
		2^{-0.99k} 
		\leq \sum_{k\in \mathbb{N} \atop k\geq \log^3(n)} 
		\sum_{f\in \mathcal{F}: \atop |f|=k}
		e^{-0.5k} \\
		& \leq \sum_{k\in \mathbb{N} \atop k\geq \log^3(n)} 
		\sum_{f\in \mathcal{F}: \atop |f|=k}
		e^{-0.5\log^3(n)} 
		\leq n^{\log(n)}
		e^{-0.5\log^3(n)}  = o(1) . \qedhere
	\end{align*}
\end{proof}

In the following lemma we consider a Waiter-Client game played on a graph $G$, i.e.~the players select edges from $E(G)$. We denote by $N_G[A]$ the set of common neighbors of vertices in $A\subseteq V(G)$ in a graph $G$; more precisely, $N_G[A]:= \left(\bigcap_{v\in A} N_G(v)\right)\setminus A$.

\begin{lemma}\label{lem:large.common.degree}
	Let $\beta\in (0,1)$. Then for every large enough integer $n$ 
	and every $t\in\mathbb{N}$ such that $t\leq 0.1\log_2(n)$ 
	the following holds. 
	Suppose $G$ is a graph on $n$ vertices and 
	for every set $A$ of $t$ vertices 
	we have a set $Y_A\subset N_G[A]$ of at least $\beta n$ common neighbors.
	Then in the Waiter-Client game on $G$, Waiter has a strategy such that at the end of the game, Client's graph $C$ 
	satisfies the following:
	$$
	|N_C[A]\cap Y_A| \geq \frac{\beta n}{200^{t+1}} \quad 
		\text{ for every }A\subset V(G)~ \text{ such that }|A|=t.
	$$
\end{lemma}

\begin{proof}
Before the start of the game fix an orientation of the edges of $G$ 
such that 
$$
|N^+[A]\cap Y_A| \geq (1-o(1)) \left( \frac{1}{2} \right)^{t} \beta n
	\geq (\beta-o(1)) n^{0.9}
$$
holds for every $A\subset V(K_n)$ of size $|A|=t\leq 0.1\log_2(n)$, 
where $N^+[A]$ denotes the set of all vertices
which are common outneighbors of the vertices in $A$. 
The existence of such an orientation can be proven with 
a simple probabilistic argument.
Moreover, 
fix any ordering of the vertices $v_1,\ldots, v_n$.
	
Waiter now plays the game in $n$ stages, where in the $i^{\text{th}}$ 
stage, she offers all the outgoing edges at the vertex $v_i$ in a suitable way.
Set $N_0^+[A]:=N^+[A]\cap Y_A$ for every $A\subseteq V$ of size $t$. 
After Stage $i$, let
$$
N_i^+[A] :=
\begin{cases}
	N_{i-1}^+[A], & ~ \text{ if }v_i\notin A \\
	\{v\in N_{i-1}^+[A]:~ v_iv~ \text{is claimed by Client 
		in Stage $i$}\}, & ~ \text{ if }v_i\in A \, .
\end{cases}
$$
We claim that Waiter can play in such a way that
for every $i\in[n]\cup \{0\}$, and after Stage $i$ we have
\begin{equation}\label{eq:inductive.bound}
		|N_i^+[A]| \geq \frac{|N_0^+[A]|}{100^k} \quad 
			\text{ if } ~ |A\cap \{v_1,\ldots,v_i\}|=k
\end{equation}
for every $A\subseteq V$ of size $t$. To prove this, we apply induction.
	
The induction start ($i=0$) is trivial.
So, let $i\geq 1$ and assume that \eqref{eq:inductive.bound} 
is correct after Stage $i-1$ for all sets $A$ of size $t$. 
Note that then for every $A$ of size $t$ we have
$$
|N_{i-1}^+[A]| \geq \frac{|N_0^+[A]|}{100^{|A|}} 
	\geq \frac{(\beta-o(1)) n^{0.9}}{100^{0.1\log_2(n)}}> n^{0.1}.
$$
For Stage $i$ notice the following. 
Claiming edges between $v_{i}$ and $N^+(v_{i})$ can be represented 
by claiming vertices in $N^+(v_{i})$;
and then claiming a certain amount of elements in some set 
$N_{i-1}^+[A]$ represents claiming the same number of edges between 
$v_{i}$ and $N_{i-1}^+[A]$. We thus can apply the strategy from 
Lemma~\ref{lemma:games.on.hypergraph} with the hyperedges being the 
sets $N_{i-1}^+[A]$ for which $v_{i}\in A$ holds. 
We claim that Property~\eqref{eq:inductive.bound} for $i$ follows from 
the lemma. 
Indeed, assume that $|A\cap \{v_1,\ldots,v_{i-1}\}|=k$.
If $v_{i}\notin A$, then by definition of $N_{i}^+[A]$ and induction, 
we have $|N_{i}^+[A]|=|N_{i-1}^+[A]|\geq \frac{|N_0^+[A]|}{100^k}$ 
while $k=|A\cap \{v_1,\ldots,v_{i-1},v_{i}\}|$. 
If otherwise $v_{i}\in A$, then 
$|A\cap \{v_1,\ldots,v_{i-1},v_{i}\}|=k+1$, and
by induction and because of Lemma~\ref{lemma:games.on.hypergraph}, 
we have
$|N_{i}^+[A]| \geq \frac{|N_{i-1}^+[A]|}{100}
	\geq  \frac{|N_0^+[A]|}{100^{k+1}}$.
Hence, the induction is complete.
	
	\smallskip
	
To finish the argument, note that by the end of the game we have
$$
|N_n^+[A]| \geq \frac{|N_0^+[A]|}{100^{|A|}} 
	\geq \frac{(1-o(1)) \left( \frac{1}{2} \right)^{t} \beta n}{100^t}
	\geq\frac{\beta n}{200^{t+1}} 
$$
for every $A$ of size $t$,
and observe that the vertices in $N_n^+[A]\subseteq Y_A$ 
belong to the common neighborhood of $A$ in Client's graph 
$C$ at the end of the game.
\end{proof}

\subsection{Proof of Theorem~\ref{thm:upperbound.cw}}

In the following we prove
Theorem~\ref{thm:upperbound.cw} with $c:=20$.
Given any large enough $n\in\mathbb{N}$,
consider a tree $T$ on $n$ vertices with a vertex $v$ of degree 
$\deg_T(v)=t:=\lfloor 0.1\log_2(n) \rfloor$
such that each vertex of $T$ has a neighbor in
$N_T(v)$, and such that each vertex in $N_T(v)$ 
has roughly the same degree, i.e.~about 
$\frac{n}{0.1\log_2(n)}<\frac{cn}{\log(n)}$.
By switching the roles of the colors blue and red,
and by applying Lemma~\ref{lem:large.common.degree}
(with $\beta=\frac{1}{2}$ and $N_A:=V\setminus A$ 
for every set $A$ of size $t$),
Waiter can make sure that the final graph of blue edges
satisfies that every vertex set of size $t$
has a common neighbor. But then, in the red complement
there cannot be a copy of $T$. \hfill $\Box$


\section{Concluding remarks}\label{sec:concluding}


We studied the asymptotics of the the function $D(n)$, where $D(n)$ is 
the largest integer $t$ such that for every tree $T$ with $v(T)=n$ and $\Delta(T)\le t$ Waiter has a winning strategy in $\WC(n,T)$. Our linear lower and upper bounds on $D(n)$ differ by a multiplicative constant. It would be interesting to now if any ``proper'' constant exists.

\begin{conjecture}
There exists a constant $c$ such that $D(n)=cn+o(n)$. 
\end{conjecture}

Going further, we could ask about the above constant $c$, however we do not dare to guess. Theorems~\ref{thm:giventree} and \ref{thm:upperbound.wc} imply that $1/3\le c<1/2$. 

\subsection{Maximum of maximum degrees of spanning trees}
A related problem is to determine $\bar D(n)$, where $\bar D(n)$ denotes 
the largest $t$ with the property that there exists a tree $T$ with $v(T)=n$ and $\Delta(T)=t$ such that Waiter has a winning strategy in $\WC(n,T)$. 
A result of Beck~\cite{beck2008combinatorial} on discrepancy games mentioned in the introduction implies that 
$\bar D(n)\le n/2+O(\sqrt{n\log n})$. On the other hand, one can show that $\bar D(n)\ge n/2+\eps \sqrt{n\log n}$ for a positive constant $\eps$ and big enough $n$ in the following way.

The above mentioned theorem of Beck is much stronger than we state it. In fact the second term order $\sqrt{n\log n}$ is optimal. More precisely, there exists a positive constant $\eps'$ such that for every big enough $t$ Waiter, playing on $K_{t,t}$, can force Client to build a red star of size $t/2+\eps' \sqrt{t\log t}$ (\cite{beck2008combinatorial}, Theorem 18.4). Let us apply it to a spanning tree game on $K_n$. We begin by partitioning $V(K_n)$ into two roughly equal sets $A,B$. First Waiter offers the edges $E(A,B)$ of the complete bipartite graph so that she forces creating a red star of size $n/4+\eps \sqrt{n\log n}$. We can assume that this star has the centre $u\in A$. 
Then Waiter forces greedily a red star of size $\lfloor (|A|-1)/2\rfloor$ with its centre $u$, by selecting edges from $E(u,A)$. In the next stage, she applies Lemma~\ref{lem:Ham.path.ends.fixed} to force a red path on the vertex set $A\setminus\{u\}$. Finally, again by Lemma~\ref{lem:Ham.path.ends.fixed}, she can force a red path on the vertex set $B$. Thereby Waiter forces Client to build a spanning connected graph in $K_n$, with the maximum degree at least $n/2+\eps \sqrt{n\log n}$. 

Summarizing, $\bar D(n)$ does not differ much from the size of the greatest red star Waiter can force in $K_n$. We conclude that there exist positive constants $\eps,\eps'$ such that for every sufficiently big $n$ we have
$$\frac n2+\eps \sqrt{n\log n}\leq \bar D(n)\leq \frac n2+\eps' \sqrt{n\log n}.$$

\subsection{Maximum degrees of spanning trees in Client-Waiter games}
We can define functions analogous to $D(n)$ and $\bar D(n)$, for Client-Waiter games. 
More precisely, let $D_{CW}(n)$ be the largest integer $t$ such that for every tree $T$ with $v(T)=n$ and $\Delta(T)\le t$ Client can build a red copy of $T$ in $K_n$. By $\bar D_{CW}(n)$ we denote 
the largest maximum degree of a spanning tree Client can build in $K_n$. Theorem~\ref{thm:upperbound.cw} implies that $D_{CW}(n)=O(n/\log n)$. The key idea behind this estimation is the fact that Waiter can play so that for every vertex set $A$ of roughly $\log n$ vertices there is a vertex connected with blue edges with all elements of $A$. This idea cannot be applied for sets $A$ of order greater than $\log n$, what can be verified for example by the analogue of the Erd\H{o}s-Selfridge Criterion for Client-Waiter games (where Client plays the role of Breaker). We believe that the order $n/\log n$ of our upper bound for $D_{CW}(n)$, supported also by the probabilistic intuition, is optimal. 

\begin{conjecture}
There exists a constant $\eps>0$ such that 
$D_{CW}(n)\ge \eps \frac{n}{\log n}$ for every sufficiently large~$n$. 
\end{conjecture}

As for $\bar D_{CW}(n)$, we have an upper bound $\bar D_{CW}(n)\le n/2+O(\sqrt{n})$, due to Beck's results on discrepancy games in the Client-Waiter version (where Waiter plays int the role of Balancer; see Theorem 18.5(a) in \cite{beck2008combinatorial}). 
We do not have any interesting lower bound on $\bar D_{CW}(n)$. We wonder if it differs much from the size of the greatest red star Client can build in $K_n$.

\subsection{Star forest games}
Finally, let us comment on star forest games on $K_n$. In Section~\ref{sec:rooted.forests}
we proved that if $F$ is a forest with $m\le n$ edges such that every its component has less than $m/3$ edges, then Waiter can force a red copy of $F$ in~$K_n$. In Section~\ref{sec:avoiding}
we proved that if $F$ is a forest consisting of two stars of size, roughly, $0.499 n$ each, then Client can avoid building a red copy of $F$ in $K_n$. It would be interesting to study the relation between the degree sequence of stars in a star forest $F$ and the outcome of Waiter-Client or Client-Waiter games such that a player tries to obtain a red copy of $F$ in $K_n$. Lemma~\ref{lem:starind} is a step in this research direction.  


\bibliographystyle{amsplain}
\bibliography{references}

\bigskip

\appendix
\section{Forcing trees with at most two roots and bounded maximum degree}

In this section we show how to modify the proof of 
Theorem 5.2 in~\cite{clemens2020fast} in order to obtain
Lemma~\ref{lem:doubleroot2}. For this, we first state a useful lemma
from~\cite{clemens2020fast}, which is a variant of Lemma 2.1 in~\cite{krivelevich2010embedding}, and a simple corollary of it.

\begin{lemma}[Lemma 5.1~in~\cite{clemens2020fast}]
\label{lem:leaves.and.paths.clemens}
	For every $\mu'\in(0,\frac{1}{2})$ there exists 
	$\alpha>0$ such that the following holds
	for every large enough integer $n$.
	If $T$ is a tree on $n$ vertices such that every vertex 
	is adjacent with at most $\alpha\sqrt{n}$ leaves in $T$,
	then $T$ contains a leaf matching of size  at least 
	$\mu'\sqrt{n}$ or a bare path of length at least $\mu'\sqrt{n}$.
\end{lemma}

\begin{corollary}
\label{cor:leaves.and.paths.forest}
	Let $\mu\in (0,\frac{1}{10})$.
	Then there exists a constant $\alpha>0$
	such that the following holds for every large enough 
	integer $n$. If $F$ is a forest without isolated vertices
	such that $v(F)\geq (1-\mu)n$ and
	$\Delta(F)\leq \alpha\sqrt{n}$ hold,
	then $F$ has a leaf matching of size at least $\mu\sqrt{n}$
	or a bare path of length at least $\mu\sqrt{n}$.
\end{corollary}

\begin{proof}
Given $\mu$, we set $\mu'=4\mu$
and apply Lemma~\ref{lem:leaves.and.paths.clemens}
to obtain a constant $\alpha$ as output.
Whenever needed, we assume that $n$ is 
large enough.

Let a forest $F$ be given as described by the 
statement of the lemma. Construct a tree $T$ from $F$
by adding one new vertex $v$ which is made adjacent 
to exactly one vertex of each component of $F$.
Note that then no leaf of $T$ is a neighbor of $v$.
Hence, by Lemma~\ref{lem:leaves.and.paths.clemens}
and since $\Delta(F)\leq \alpha\sqrt{n}$,
$T$ has a leaf matching of size $\mu'\sqrt{v(T)}$
or a bare path of length at least $\mu'\sqrt{v(T)}$.
Note that
$$
\mu'\sqrt{v(T)}
\geq 
4\mu\sqrt{(1-\mu)n}
\geq 
2\mu\sqrt{n}.
$$
Hence, after removing $v$,
there must be 
a leaf matching of size
at least $\mu\sqrt{n}$
or a bare path of length at least
$\mu\sqrt{n}$.
\end{proof}

Let $T$ be a tree as given by the assumption of 
Lemma~\ref{lem:doubleroot2}.
In order to force a copy of the desired tree $T$ fast, the main idea will be to first embed everything of $T$ except from a special structure, i.e.~a long bare path or a large leaf matching,
in a careful greedy way, and then to finish the tree by forcing this remaining structure. Now, in case that there are $r=2$ roots $x,y$ with some path $P_{x,y}$ between them, we also need to
make sure to claim an appropriate copy of $P_{x,y}$ between the fixed images of $x,y$. Depending on the structure of $P_{x,y}$, we may first create this copy of $P_{x,y}$ and then continue as in~\cite{clemens2020fast}, or find a long bare path contained in $P_{x,y}$ which can be used as the special structure to be embedded last.
Because of this, we will deal with different cases in the proof of Lemma~\ref{lem:doubleroot2}, but in each of these cases at some point the main part of the strategy will be purely identical with the strategy presented in~\cite{clemens2020fast}.
In the proof of Lemma~\ref{lem:doubleroot2}, at the end of this section, we will mostly discuss our modifications and how to deal with $P_{x,y}$, while the following preparatory lemma will be used to cover those parts of the main strategy which are identical with~\cite{clemens2020fast}.

\begin{lemma} \label{lem:tree.extensions}
	For every $\mu>0$ there exists $\alpha>0$ such that the following 
	holds for every large enough integer $n$.
	Let a Waiter-Client game on $K_n$ be in progress.
	Let $T$ be a tree on $n$ vertices with 
	$\Delta(T)\leq \alpha\sqrt{n}$, and let $F\subset T$ be an 
	induced subgraph of $T$. Assume that the following holds:
	\begin{enumerate}[(a)]
	\item\label{barepath}
        $F$ consists of at most two components, 
		and if the number of components is two, then
		in $T$ they are connected by a bare path of length 
		at least $\mu\sqrt{n}$.
	\item\label{ftree} 
        If $F$ is a tree, then $T_0:=T\setminus E(F)$ 
		contains at least one of the following:
		\begin{enumerate}[(b1)]
		\item\label{b1} a bare path of length $\mu\sqrt{n}$ which 
			uses no vertex of $F$, 
		\item\label{b2} a matching of size $\mu\sqrt{n}$ such 
			that each of its edges intersects $L(T)$.
		\end{enumerate}
	\item\label{partial} 
        So far, Waiter forced a copy $\bar F$ of $F$ 
		such that every colored edge is contained in $V(\bar F)$. 
		Let $f: V(F)\rightarrow V(\bar F)$ be the embedding of $F$ 
		into Client's graph.
	\end{enumerate}
	Then Waiter can force a copy $\bar T$ of $T$ within 
	$e(T_0)+1$ further rounds such that for every $v\in V(F)$
	the copy of $v$ in $\bar T$ is the fixed vertex $f(v)$.
\end{lemma}

\begin{proof}
Without loss of generality assume that $\mu\leq \frac{1}{3}$.
Lemma~\ref{lem:tree.extensions} can be proven
along the lines of Theorem~5.2 in~\cite{clemens2020fast}.
In particular, we choose $\alpha<\frac{\mu}{20}$ 
as in the proof of  Theorem~5.2 in~\cite{clemens2020fast}.
Moreover, we will use the notation 
from~\cite{clemens2020fast,ferber2012fast},
similarly to the proof of Lemma~\ref{lem:doubleroot1deg}. 
Let $T'\subset T$ be any subgraph of $T$, and
let $S \subseteq V(T')$ be an arbitrary set. 
Then an $S$-\textit{partial embedding} of $T'$ into
some graph $G\subset K_n$ is an injective mapping 
$f: S \rightarrow V(G)$ such that $f(x)f(y)\in E(G)$ holds if 
$xy\in E_{T'}(S)$. We call a vertex $v \in f(S)$ \textit{closed} with
respect to $T'$ if all the neighbors of $f^{-1}(v)$ in $T'$ are 
embedded by $f$. If $v\in f(S)$ is not closed with respect to ~$T'$, then we 
call it \textit{open} with respect to  $T'$, and moreover, 
we write $\mathcal{O}_{T'} = \mathcal{O}_{T'}(S)$ for the set of all 
vertices that are open with respect to  $T'$. The vertices not used for the 
embedding are collected in the set $A:=V(K_n) \setminus f(S)$, 
and they are called \textit{available}.

We consider two cases, depending on the number of components of $F$.

\paragraph{\textbf{Case 1}} 
Assume that $F$ has two components. Then by \ref{barepath} there is a 
bare path $P$ in $T$ of length at least $\mu\sqrt{n}$ which 
connects these two components.

Denote with $u$ and $w$ the endpoints of $P$, 
and let $u_1$ and $w_1$ be the neighbors of $u$ and $w$ in $P$, 
respectively. Note that $T-(V(P)\setminus\{u,w\})$ is a forest 
with two tree components $T_1$ and $T_2$. 
We let $T'\subset T$ be the forest induced by 
$E(T_1)\cup E(T_2)\cup \{uu_1,ww_1\}$, and note that $F\subset T'$. 

The main idea of Waiter's strategy now is to 
(1) first extend $\bar F$ to a copy $\bar{T'}$ of $T'$ 
without wasting any move and then 
(2) to obtain $\bar T$ by forcing an appropriate copy of 
$P-\{u,w\}$ while wasting at most one move. 
Such a strategy is already given in  Case A of the proof 
of Theorem 5.2 in~\cite{clemens2020fast},
and the same strategy can be applied here.
During part (1) of this strategy, Waiter maintains a set $S$ with 
$V(F)\subseteq S\subseteq V(T')$ as well as an $S$-partial embedding $f$ 
of $T'$ into Client's graph, which represents the subgraph of $T'$ 
which has been forced so far. Moreover, Waiter makes sure that at any 
time $f(v)=g(v)$ holds for every $v\in V(F)$.

Initially, we have $S = V(F)$ and the $S$-partial embedding $f$ is 
given by property \ref{partial}. For step (1), Waiter can play
the same strategy as for Stage II in Case A of the proof of 
Theorem~5.2 in~\cite{clemens2020fast}, 
where the only modification is that $e(T')-e(F)$ rounds get played; 
and for  step (2), she can play
exactly as in the strategy of Stage III in Case A of the proof of 
Theorem~5.2 in~\cite{clemens2020fast}, where only one round gets 
wasted. For this note that the heart of this strategy 
is Observation~5.3 in~\cite{clemens2020fast}.
This observation lists two properties (i) and (ii), that can be 
guaranteed as long as Waiter can follow the strategy of Stage~II, 
provided these properties were true at the beginning of this stage.
Moreover, in the discussion afterwards in~\cite{clemens2020fast} 
it is shown that, as long as these properties are maintained, 
Waiter can follow the strategy of Stage II and then also succeed with 
Stage~III, until a copy of $T$ is forced as required. 

It thus follows that, in order to guarantee that the same strategy can 
be used here, we only need to verify that the properties
(i) and (ii) from Observation 5.3 in~\cite{clemens2020fast} hold at 
the beginning, i.e.~when Waiter starts to extend the embedding of $F$ 
towards $T'$. The property (i) requires that
$|A|\geq \mu \sqrt{n} - 2$ and $e_{C\cup W}(A)=0$ hold.
The first part is satisfied at the beginning, since 
then $A=V(K_n)\setminus V({\bar F})$
and hence $|A| = n - |V(F)| \geq |V(P)| - 2 \geq \mu\sqrt{n} - 2$;
the second part holds because of \ref{partial}.
Moreover, property (ii) requires that $\deg_W(x,A)\leq \deg_C(x)$ 
for every $x\in f(S)$, which holds since by \ref{partial} 
there are no colored edges incident with $A$.

\paragraph{\textbf{Case 2}}
Assume that $F$ has exactly one component. Then by \ref{ftree}
we know that $T_0=T\setminus E(F)$ contains 
a bare path as described in \ref{b1} or a matching as described in \ref{b2}.

In the first case of having a bare path 
Waiter can play essentially the same way as in Case 1.
So, we may assume from now on that in $T_0$
there is a matching $M$ of size $\mu\sqrt{n}$ such 
that each of its edges intersects $L(T)$.
In particular, there is matching $M'\subset M$ of size at least 
$\frac{1}{2}\mu\sqrt{n}$
such that its edges have pairwise distance at least 2.
Let $L':=V(M')\cap L(T)$,
set $T'=T-L'$ and note that $F\subset T'$.

The main idea of Waiter's strategy now is to (1) first extend 
$\bar F$ to a copy $\bar{T'}$ of $T'$ without wasting any move and 
then (2) to obtain $\bar T$
by forcing an appropriate copy of $M'$ while wasting at most one move. 
Such a strategy is already given in 
Case B of the proof of Theorem 5.2 in~\cite{clemens2020fast}.
During part (1) of this strategy, Waiter again maintains 
a set $S$ with $V(F)\subseteq S\subseteq V(T')$ as well as an 
$S$-partial embedding $f$ of $T'$ into Client's graph, which 
represents the subgraph of $T'$ which has been forced so far. 
Moreover, Waiter makes sure that at any time $f(v)=g(v)$ holds 
for every $v\in V(F)$.

Initially, we have $S = V(F)$ and $f$ is given by property \ref{partial}. 
For step (1), Waiter can play exactly the same strategy as for 
Stage~II in Case~B of the proof of Theorem 5.2 
in~\cite{clemens2020fast}, where the only modification is that 
$e(T')-e(F)$ rounds get played; and for  step (2), she can play
exactly as in the strategy of Stage~III in Case~B of the proof of 
Theorem~5.2 in~\cite{clemens2020fast}, where only one round 
gets wasted. For this note that the heart of this strategy 
is Observation~5.4 in~\cite{clemens2020fast} 
(where $\alpha$ is replaced with $\eps$).
This observation lists four properties (i)--(iv), that can be 
guaranteed as long as Waiter can follow the strategy of Stage II, 
provided these properties were true at the beginning of this stage.
Moreover, in the discussion afterwards it is shown that, as long as 
these properties are maintained, 
Waiter can follow the strategy of Stage~II and then also succeed with 
Stage III, until a copy of $T$ is forced as required. The first two 
properties are the same as in Case 1 above, and they can be 
verified analogously. Properties (iii) an (iv) consider the set 
$S':=S\cap N_T(L')$, and they are satisfied if
$\deg_W(x,f(S'))\leq 1$ for every $x\in A$
and $e_W(f(S'),A) \leq \alpha \sqrt{n}$.
Both inequalities are guaranteed at the beginning, since by \ref{partial}
all colored edges are contained in $V(\bar F)$.
\end{proof}

\medskip

\begin{proof}[Proof of Lemma~\ref{lem:doubleroot2}]
Let $\mu\in (0,\frac{1}{20})$.
Apply Corollary~\ref{cor:leaves.and.paths.forest} and
Lemma~\ref{lem:tree.extensions} with input $2\mu$
to obtain outputs $\alpha_{\ref{cor:leaves.and.paths.forest}}$
and $\alpha_{\ref{lem:tree.extensions}}$,
and let $\eps>0$ be such that 
Theorem 5.2 in~\cite{clemens2020fast} holds.
We set 
$\alpha=\min\{\alpha_{\ref{cor:leaves.and.paths.forest}},	
	\alpha_{\ref{lem:tree.extensions}},\eps,\frac{\mu}{3}\}$.

If $r=1$ then the statement of Lemma~\ref{lem:doubleroot2}
follows directly from Theorem 5.2 in~\cite{clemens2020fast}.
Hence, from now on we may assume that $r=2$.
Denote the roots of $T$ with $x$ and $y$,
and let $P_{xy}$ be the unique $xy$-path in $T$. 
Moreover let $v_x$ and $v_y$ be the vertices in $K_n$ to which $x$ 
and $y$ must be mapped, respectively. We distinguish three cases.

\paragraph{\textbf{Case 1}}
Assume that at least $\mu\sqrt{n}$ inner vertices of $P_{xy}$ 
have degree at least 3 in $T$. 
Then set $F:=P_{xy}$ and notice that $T_0:=T\setminus E(F)$ 
consists of at least $\mu\sqrt{n}$ non-trivial components. 
In particular, $T_0$ must contain a matching of size $\mu\sqrt{n}$
such that each of its edges intersects $L(T)$; 
i.e. property \ref{ftree} from Lemma~\ref{lem:tree.extensions} holds.

In this case Waiter plays as follows:
let $V_1\subset V(K_n)$ be any subset 
with $v_x,v_y\in V_1$ and $|V_1|=v(P_{xy})$.
Then playing the first $e(P_{xy})+1$ rounds on
$K_n[V_1]$ according to Lemma~\ref{lem:Ham.path.ends.fixed}, 
Waiter forces a Hamilton path $H_{xy}$
with endpoints $v_x$ and $v_y$. 
Afterwards, Waiter follows the strategy from
Lemma~\ref{lem:tree.extensions} for $e(T_0)+1$ rounds, applied with
$F:=P_{xy}$ and $\bar F:=H_{xy}$. 
This way, Waiter forces a copy of $T$ as required.

\paragraph{\textbf{Case 2}}
Assume that $7\leq v(P_{xy})< \mu n$.
Then Waiter plays analogously to Case 1 with $F:=P_{xy}$. 
In order to make sure that this is possible, 
we only need to argue that
$T_0:=T\setminus E(F)$ satisfies property \ref{ftree}
from Lemma~\ref{lem:tree.extensions}. 

To do so, consider first the forest $T_0':=T-V(P_{xy})$,
and note that $v(T_0')>(1-\mu)n$.
If $T_0'$ contains at least $\alpha\mu n$
isolated vertices (which are leaves in $T_0$), 
then since $\Delta(T_0)\leq \alpha\sqrt{n}$,
we find in $T_0$ a leaf matching of size at least
$\mu\sqrt{n}$, and \ref{b2} is satisfied. 
So assume otherwise, i.e.~that are less than
$\alpha\mu n$ isolated vertices,
and let $T_0''$ be the forest obtained from $T_0'$
by deleting all isolated vertices.
Then $v(T_0'')>(1-2\mu)n$.
Hence, by Corollary~\ref{cor:leaves.and.paths.forest}
it follows that $T_0''$ contains a bare path $P'$
or a leaf matching $M'$ of size at least $2\mu\sqrt{n}$.
If we find such a bare path $P'$, 
then in $T$ at most one vertex of $P'$
can be in the neighborhood of $V(P_{xy})$.
Hence, there must be a path $P\subset P'$
with $e(P)\geq \frac{e(P')}{2}$ which is a bare path in
$T$, and \ref{b1} follows. So, assume from now on that
we find a leaf matching $M'$ as described above.
We then consider all the components $C_1,C_2,\ldots$
of $T_0''$ and denote with $v_i$ the unique vertex in $C_i$ which
is a neighbor of $V(P_{xy})$ in $T$.
Note that from each component $C_i$ we can take at least one edge
for a leaf matching $M$ in $T$ 
(this edge does not need to belong to $M'$).
Moreover, if for a component $C_i$ we have $|M'\cap E(C_i)|\geq 2$, 
then at most one edge in $M'\cap E(C_i)$ contains $v_i$,
and hence we can take all the remaining edges in
$M'\cap E(C_i)$ for a leaf matching $M$ in $T$.
It thus follows that there is a leaf matching $M$ in 
$T$ with $|M|\geq \frac{|M'|}{2}\geq \mu\sqrt{n}$,
and thus \ref{b2} is satisfied.

\paragraph{\textbf{Case 3}}
Assume that less than $\mu\sqrt{n}$ inner vertices of $P_{xy}$ 
have degree at least 3 in $T$, and that $v(P_{xy})\geq \mu n$.
Then we can split $P_{xy}$ into at least $\mu\sqrt{n}$ edge-disjoint 
intervals each of length at least $\mu\sqrt{n}$,
and hence, by the assumption of this subcase, at least one of these 
intervals must form a bare path $P'$. Label the endpoints of $P'$ with 
$x'$ and $y'$. Let $F$ be the forest obtained from $T$ 
by removing all inner vertices of $P'$, and 
note that $F$ has two components such that each of them contains 
one of the roots $x,y$. Moreover,
$$v(F)+\Delta(F) \leq n - (v(P')-2) + \Delta(T)  \leq n + 2 
	- \mu\sqrt{n} + \alpha \sqrt{n} < n\, .$$ 

Waiter now plays as follows.
At first she plays according to the strategy of 
Lemma~\ref{lem:doubleroot1deg} (with $d:=0$ and $B:=K_n$) 
until after $e(F)$ rounds a copy $\bar{F}$ of $F$ is forced such
that $x,y$ are mapped to $v_x,v_y$,
and such \ref{doubleroot1deg_nonleaf}  and \ref{doubleroot1deg_bluestars}  from Lemma~\ref{lem:doubleroot1deg} hold.
Let $v_{x'}$ and $v_{y'}$ be the images of $x'$ and $y'$. 
Waiter continues as follows:
in the next two rounds, Waiter first forces 
an edge $x'x''$ with $x''\notin V(\bar{F})$ 
(by offering two such edges), and then forces an edge
$y'y''$ such that $y''\notin V(\bar{F})\cup \{x''\}$ 
(by offering two such edges).
This is possible by \ref{doubleroot1deg_bluestars}  from Lemma~\ref{lem:doubleroot1deg} and 
since
$|V(K_n)\setminus V(\bar{F})|\geq \mu\sqrt{n}-1 
	>\alpha\sqrt{n} + 4\geq \Delta(T)+4$.
Afterwards, notice that by property \ref{doubleroot1deg_nonleaf}  from 
Lemma~\ref{lem:doubleroot1deg}, all edges in
$V(K_n)\setminus V(\bar{F})$ are still uncolored.
Now, playing only on $V(K_n)\setminus V(\bar{F})$
according to Lemma~\ref{lem:Ham.path.ends.fixed},
Waiter forces a Hamilton path between $x''$ and $y''$, 
and finishes a copy of $T$ as desired.
\end{proof}

\end{document}